\documentclass[11pt,reqno]{amsart}
\usepackage{amsmath}
\usepackage{authblk} 

\usepackage{algorithm}
\usepackage[noend]{algpseudocode}
\usepackage[utf8]{inputenc}

\usepackage{bm}
\usepackage{graphicx}
\usepackage{subcaption}
\usepackage[margin = 3cm]{geometry}
\usepackage{enumerate} 
\usepackage{parskip}
\usepackage{hyperref}
\usepackage{comment}
\usepackage{enumitem}
\usepackage{xcolor}
\usepackage{booktabs}
\usepackage{multirow}

\usepackage[foot]{amsaddr}
\allowdisplaybreaks

\newtheorem{theorem}{Theorem}[section]
\newtheorem{lemma}[theorem]{Lemma}

\newtheorem{proposition}[theorem]{Proposition}
\theoremstyle{definition}

\newtheorem{example}[theorem]{Example}

\theoremstyle{remark}
\newtheorem{remark}[theorem]{Remark}
\numberwithin{equation}{section}

\newcommand{\C}{\mathcal{C}}

\newcommand{\M}{\mathcal{M}}

\newcommand{\R}{\mathbb{R}}

\newcommand{\Pb}{\mathcal{P}}
\newcommand{\I}{\mathcal{I}}

\DeclareMathOperator*{\argmin}{arg\!\min}

\numberwithin{equation}{section}
\renewcommand{\theequation}{\arabic{section}.\arabic{equation}}
\newcommand{\red}[1]{\textcolor{red}{#1}}

\newcommand{\lnote}[1]{{\color{red}\textbf{Levon:} #1}}
\newcommand{\hnote}[1]{{\color{blue}\textbf{HLiu:} #1}}
\begin{document}

\makeatletter
\renewcommand{\email}[2][]{%
  \ifx\emails\@empty\relax\else{\g@addto@macro\emails{,\space}}\fi%
  \@ifnotempty{#1}{\g@addto@macro\emails{\textrm{(#1)}\space}}%
  \g@addto@macro\emails{#2}%
}
\makeatother

\title[AEPG] 
{Adaptive Preconditioned Gradient Descent with Energy}



\author{Hailiang Liu$^\dagger$}
\address{$^\dagger$Department of Mathematics, Iowa State University, Ames, IA 50011}
\email{hliu@iastate.edu}

\author{Levon Nurbekyan$^\ddagger$}
\address{$^\ddagger$Department of Mathematics, Emory University, Atlanta, GA 30322}
\email{lnurbek@emory.edu}

\author{Xuping Tian$^{\dagger*}$}

\author{Yunan Yang$^\mathsection$}
\address{$^\mathsection$Department of Mathematics, Cornell University, Ithaca, NY 14853}
\email{yunan.yang@cornell.edu}

\thanks{$^*$corresponding author: \url{xupingt@iastate.edu}}


\begin{abstract}
We propose an adaptive step size with an energy approach for a suitable class of preconditioned gradient descent methods. We focus on settings where the preconditioning is applied to address the constraints in optimization problems, such as the Hessian-Riemannian and natural gradient descent methods. More specifically, we incorporate these preconditioned gradient descent algorithms in the recently introduced Adaptive Energy Gradient Descent (AEGD) framework. In particular, we discuss theoretical results on the unconditional energy-stability and convergence rates across three classes of objective functions. Furthermore, our numerical results demonstrate excellent performance of the proposed method on several test bed optimization problems.
\end{abstract}

\maketitle

\section{Introduction}\label{Intro}
We present an optimization method involving adaptive step size with energy coupled  with preconditioned gradient descent, designed for solving the constrained optimization problem: 
\begin{align}\label{minp}
\min_{\theta\in \Theta} L(\theta).
\end{align}
Here $L\in C^1(\R^n)$ is bounded below, $\Theta \subset \mathbb{R}^n$ represents the set of all possible parameters.  In scenarios where $\Theta = \mathbb{R}^n$,  one common strategy for solving~\eqref{minp} is to  apply simple iterative algorithms such as the standard gradient descent (GD):
\begin{equation}\label{eq:gd}
\theta_{k+1}=\theta_k -\eta \nabla L(\theta_k),
\end{equation}
where $\eta>0$ is the step size. However, if $\Theta$ represents a constraint set where the simple GD cannot guarantee the iterates to stay in $\Theta$, the projected gradient descent method (PGD) is often used:
\begin{equation}\label{eq:proj_gd}
\theta_{k+1}=\mathcal{P}_\Theta (\theta_k -\eta \nabla L(\theta_k)),
\end{equation}
where $\mathcal{P}_\Theta$ denotes the projection onto $\Theta$. Under relatively mild assumptions on $L$ and $\Theta$, it can be proven that using (projected) GD with a sufficiently small step size $\eta$ enhances the quality of the iterates over time, $L(\theta_{k+1})< L(\theta_k)$, and the method converges to a stationary point within  $\Theta$~\cite{Po87}. A more refined version of the algorithm determines $\eta$ by searching for the minimum of the objective function along the descent direction~\cite{NW06} at each iteration.


Closely related classes of algorithms are preconditioned gradient descent methods defined as
\begin{equation}\label{eq:pre-gd}
    \theta_{k+1}=\theta_k -\eta T_k \nabla L(\theta_k),
\end{equation}
where $T_k$ is a suitable preconditioning matrix. Algorithms of the form \eqref{eq:pre-gd} include many classical and well-known optimization methods, such as Newton's and Gauss--Newton methods~\cite{NW06}, Hessian--Riemannian gradient descent (HRGD)~\cite{ABB04}, and the natural gradient descent (NGD)~\cite{A98}, just to name a few. In this work, we are interested in the last two in the context of constrained optimization problems. More specifically, we think of the HRGD as a method for enforcing \textit{convex constraints} and the NGD as a method of incorporating the geometry of the \textit{constraints manifold} when the latter admits an \textit{explicit parametrization}.

The HRGD corresponds to the choice $T_k=P(\theta_k) \left( \nabla^2 h(\theta_k)\right)^{-1}$, where $h$ is a suitable convex function, and $P(\theta_k)$ is the $\nabla^2 h(\theta_k)$-orthogonal projection onto an appropriate linear subspace. Furthermore, NGD corresponds to the choice $T_k=G(\theta_k)^{-1}$, where $G(\theta)$ is a suitable positive semidefinite matrix~\cite{NLY22}.

Although our focus in this paper is on the HRGD and NGD, our proposed approach is versatile and can be extended to problems beyond HRGD and NGD, provided a well-identified preconditioning matrix is available. For instance, it can be applied to the Laplacian Smoothing Gradient Descent (LSGD) method  \cite{OW19}, where the authors use $A=I-\sigma \Delta$ to reduce the variance of the SGD method, with $I$ being the $n\times n$ identity matrix and $\Delta $ the discrete one-dimensional Laplacian.

Adaptive Energy Gradient Descent (AEGD)~\cite{LT20,LT21} is a recently introduced method for enhancing the performance of gradient descent algorithms for unconstrained problems. More specifically, let $c \in \mathbb{R}$ be such that $\inf\limits_{\theta \in \mathbb{R}^n}\left( L(\theta)+c\right)>0$, and $l(\theta)=\sqrt{L(\theta)+c}$. The algorithm initiates from $\theta_0$ and $r_0=l(\theta_0)$ and iterates as follows: 
\begin{subequations}\label{vv}
\begin{align}
v_k &= \nabla l(\theta_k),\\
r_{k+1} &= \frac{r_k}{1+2\eta\|v_k\|^2},\\
\theta_{k+1} &= \theta_k - 2\eta r_{k+1}v_k.
\end{align}
\end{subequations}
Here, \{$r_k^2\}_{k=0}^\infty$ play the role of an energy. In sharp contrast to GD, energy-adaptive gradient methods, like AEGD, exhibit unconditional energy stability; that is, $\{r_k^2\}_{k=0}^\infty$ is a decreasing sequence irrespective of the base step size $\eta>0$. The excellent performance of AEGD-type schemes has been demonstrated across various optimization tasks \cite{LT20, LT21}. AEGD can be extended to accommodate  stochastic effects in gradient estimation \cite{LT20} or incorporate momentum for further convergence acceleration~\cite{LT21}. Moreover, results in~\cite{LT20} indicate that $c$ has minimal impact on the performance of~\eqref{vv}. 

The main objective of this paper is to extend the AEGD framework to include preconditioned gradient descent algorithms of the form~\eqref{eq:pre-gd}. Indeed, setting $c$ and $l$ as above, we consider the following algorithm:
\begin{subequations}\label{aeng0}
\begin{align}
& v_k= T_k 
\nabla l(\theta_k),\\
& r_{k+1} =\frac{r_k}{1+2\eta \|v_k\|^2},\\
& \theta_{k+1}=\theta_k-2\eta r_{k+1}v_k,
\end{align}
\end{subequations}
which we call the adaptive preconditioned gradient descent with energy (AEPG).

For the HRGD, we consider only affine constraints $\Theta$, whereas for the NGD, we consider only $\Theta=\mathbb{R}^n$. In both cases, the updates in~\eqref{aeng0} guarantee that $\theta_k \in \Theta$ holds for all $k \in \mathbb{N}$ as long as $\theta_0 \in \Theta$, with any fixed $\eta$. Under these scenarios, the proposed method~\eqref{aeng0} demonstrates  unconditional energy stability,  regardless of the magnitude of $\eta$.   This raises the following question: 

\begin{center}
{\sl What is the convergence rate of this preconditioned gradient descent method with an adaptive time step incorporating energy?}
\end{center}

We provide answers for two scenarios:
\begin{enumerate}[leftmargin=*]
    \item When $\Theta=\mathbb{R}^n$, and $T_k=A_k^{-1}$ for some symmetric and positive definite $A_k$, we prove that the AEPG method converges to a  minimum of the objective function with an appropriate step size. The convergence rates are influenced by the geometric properties of the objective function, and we provide  rates for both convex and non-convex objectives, including those that satisfy the Polyak–Lojasiewics (PL) condition. 

    The analysis in this case covers the HRGD without affine constraints, $P_k=I$ and $A_k=\nabla^2 h(\theta_k)$, and NGD, $A_k=G(\theta_k)$.
    \item When $\Theta=\{\theta\in \mathbb{R}^n, B\theta=b\}$, and $T_k=P_k G_k^{-1}$, where $G_k$ is symmetric and positive definite, and $P_k$ is the $G_k$-orthogonal projection onto $\operatorname{ker}(B)$, we prove that the AEPG method converges to a minimum of the objective function with a suitably small step size. We derive convergence rates for both convex and non-convex objectives, including those that satisfy a projected Polyak–Lojasiewics (PL) condition defined in (\ref{PPL}).

    The analysis in this case covers the HRGD with affine constraints, $G_k=\nabla^2 h(\theta_k)$ and $P_k=P(\theta_k)$.
\end{enumerate}

A specific example of interest is the Wasserstein Natural Gradient Descent (WNGD), which has recently gained much attention in machine learning and PDE-based optimization applications~\cite{LM18,arbel2019kernelized,li2019affine,shen2020sinkhorn,ying2021natural,NLY22}. Previous works primarily focus on  computing descent directions, whereas our emphasis here is on time steps. We apply the AEPG algorithm in Equation~\eqref{aeng0} to WNGD, incorporating the efficient computation of $\{v_k\}$ using  methods developed in~\cite{NLY22}. See Section~\ref{Wasser} for more details on WNGD.

Our approach represents a unique fusion of preconditioned gradient and energy-adaptive strategies. This combination enables acceleration of standard GD both in terms of better descent directions and step sizes,  with provable convergence rates under relatively mild conditions on both the objective function and the step size. Our current findings for constrained optimization problems provide insights into energy-driven update rules with preconditioned gradients, creating a path toward additional algorithms employing AEGD-type techniques.

\subsection{Contributions}
In summary, our main contributions can be outlined as follows.
\begin{enumerate}
    \item {\bf Novel AEPG Algorithm.} We introduce a novel Adaptive Energy-based Preconditioned Gradient (AEPG) algorithm for the optimization problem represented by  (\ref{minp}). The convergence theory 
    is established across three distinct cases: general differentiable objective functions, nonconvex objective functions satisfying the Polyak-Lojasiewics (PL) condition, and convex objective functions. 
    \item {\bf  Application to HRGD and NGD.}  We apply our proposed AEPG  to Hessian-Riemannian Gradient Descent (HRGD) and Natural Gradient Descent (NGD), two specific optimization algorithms that fit the general framework of Equation~\eqref{eq:pre-gd}. Our approach improves these methods by providing  unconditional energy stability. 
    \item {\bf Unifying Framework.}  We demonstrate that HRGD can be interpreted as an instance of NGD. To the best of our knowledge, this is a new connection, providing a unified framework for studying various constrained optimization algorithms. 
    \item {\bf Numerical Experiments.} We conduct several numerical experiments,  illustrating  that AEPG exhibits faster convergence compared to classical algorithms without using an adaptive step size. This superiority is powerful in handling challenging scenarios such as ill-conditioned or nonconvex problems.
\end{enumerate}

    
\subsection{Review of Related Literature}
It is natural to seek an improved descent direction to accelerate the convergence speed of GD. This can be achieved by leveraging curvature information from the objective function,  as demonstrated in Newton and quasi-Newton methods. Specifically, for $\alpha$-smooth objective functions, these methods attain a quadratic and superlinear local convergence rate, respectively~\cite{NW06}. The natural gradient method, initially introduced in~\cite{A98}, falls within this category, using curvature information from the Riemannian metric tensor.  
In recent years, the Wasserstein natural gradient has attracted significant attention ~\cite{LM18,LZ19, NLY22}. We refer interested readers to \cite{M20} for additional insights on the natural gradient. Adding momentum is another widely employed technique to enhance descent directions. Examples include the heavy-ball (HB) method \cite{P64} and Nesterov's accelerated gradient (NAG) \cite{N83, N04}. 

An alternative avenue to expedite GD (or SGD) convergence involves the use of adaptive step sizes. Previous steps' gradients are often used to adjust the step size. This idea is used in renowned works such as AdaGrad~\cite{DHS11} and RMSProp~\cite{TH12}. Adam~\cite{KB15},  combining  momentum and adaptive step size benefits, has demonstrated rapid convergence in early iterations and has gained widespread use in the machine learning community. Further enhancements to Adam can be explored in \cite{Do16, RKK18, LJ+19, LXLS19, ZT+20}. These adaptive methods not only update the step size in each iteration but also compute  individual step sizes for each parameter. Other adaptive techniques adjust the step size based on the error of the objective function value~\cite{A98,WD19}.

Manifold optimization is a well-explored area in the literature. The general feasible algorithm framework on the manifold involves the use of a retraction operator, with parallel translation or vector transport aiding in pulling the update back to the manifold constraint. Computational costs and convergence behaviors of different retraction operators vary significantly; interested readers can find further details in   \cite{Ga82, AMS08, Bo13, HMWY18, LB19} and related references. 

Finally, a considerable body of work exists on strategies for improving GD in unconstrained optimization problems. For line search-based GD methods, the classical Armijo rule (1966) and the Wolfe conditions (1969) guide the inexact line search, with additional insights available in \cite{ZH04, NSD21}, the book of \cite{NW06}, and associated references.   
The BB method  \cite{BB88}, a gradient method with modified step sizes motivated by Newton's method but without involving any Hessian, is noteworthy. Numerous studies explore methods capable of handling constraints, such as projected gradient descent,
proximal gradient methods, penalty methods, barrier methods, sequential quadratic programming (SQP) methods, interior-point methods, augmented Lagrangian/multiplier methods, primal-dual strategies, and active set strategies; see, for instance,  \cite{Ber96,NW06}. Therefore, our work contributes to the broader field of continuous optimization with constraints.  

\subsection{Notation}
Throughout this paper, we adopt the following notations:  $\|\cdot\|$ denotes the $\ell^2$ norm for vectors and the spectral norm for matrices. Additionally,  $\lambda_1(\cdot),\ldots, \lambda_n(\cdot)$ represent the smallest and largest eigenvalues, respectively. For a function $L$, we use $\nabla L$ and $\nabla^2 L$ to denote its gradient and Hessian, respectively, while $L^*$ denotes the global minimum value of $L$. For a positive definite matrix $A$, we introduce the vector-induced norm by $A$ as $\|w\|_A:=\sqrt{\langle w, Aw\rangle}$. The set $\{1,2, \cdots, n\}$ is denoted as $[n]$.

\subsection{Organization}

The rest of this paper is organized as follows. In Section~\ref{AENG}, we introduce  the AEPG method and present the key theoretical results, including unconditional energy stability, convergence, and convergence rates.  Section~\ref{ap} delves into HRGD, demonstrating how AEPG can solve a specific type of constrained optimization problem using  a Hessian--Riemannian metric. Moving to Section~\ref{sec:NGD}, we show that the HRGD can be cast as an NGD. In particular, Section~\ref{Wasser} discusses the Wasserstein natural gradient and an efficient numerical method for computation through AEPG.  Numerical results and examples are given in Section~\ref{numeric} to further elucidate the discussed concepts. Conclusions and further discussions follow in Section~\ref{discuss}. Technical proofs for theoretical results in Section~\ref{AENG} are presented in Appendix A.
A construction of the preconditioned matrix for HRDG is presented in Appendix B. Finally, Appendix C includes an illustrative example to explain the projected PL condition.    

%

\section{Preconditioned AEGD and theoretical results}\label{AENG}
In this section, we present the theoretical results for the  main preconditioned AEGD algorithm on stability, convergence, and convergence rates. 

\subsection{Assumptions}\label{ch-aeng} 

We assume that the objective function $L$ is differentiable and bounded from below, and the update rule for AEPG follows~\eqref{aeng0}. Throughout the paper, we assume that $\theta_0 \in \Theta$ yields
\begin{align}\label{kin}
\theta_k \in \Theta, \quad \forall k \in \mathbb{N} \quad \text{and} \quad \eta >0.  
\end{align} 
This assumption is valid in two specific cases we address:
\begin{enumerate}[leftmargin=*]
    \item When $\Theta=\mathbb{R}^n$ and $T_k=A_k^{-1}$, where the matrices $(A_k)_{k=0}^\infty$ are symmetric and uniformly positive definite; that is,
\begin{align}\label{ass}
\lambda_1\|\xi\|^2\leq \xi^\top A_k \xi \leq \lambda_n \|\xi\|^2, \quad \forall \xi \in \R^n,~ k\geq 0,
\end{align} 
for some $\lambda_n \geq \lambda_1>0$.
    \item When $\Theta=\{\theta\in \mathbb{R}^n, B\theta=b\}$ and $T_k =P_k G_k^{-1}$, where $B\in\mathbb{R}^{m\times n}$, $b\in\mathbb{R}^m$, $(G_k)_{k=0}^\infty$ are symmetric and uniformly positive definite, and $P_k$ are the $G_k$-orthogonal projection operators onto $\operatorname{ker}(B)$. The explicit form of the projection matrix is
\begin{equation}\label{AGPk+}
P_k = I- G_k^{-1} B^\top (B G_k^{-1}B^\top)^{-1}B.
\end{equation}
\end{enumerate}

In Section \ref{subsec:theory_unconstrained} we discuss the unconstrained case, whereas Section \ref{subsec:theory_constrained} is devoted to the affine constraints case.

\subsection{No equality constraints}\label{subsec:theory_unconstrained}

In this section, we assume that $\Theta = \mathbb{R}^n$, and so \eqref{kin} is valid.
\begin{theorem}[Unconditional energy stability]\label{thm1}  
AEPG \eqref{aeng0} is unconditionally energy stable. Specifically, for any step size $\eta>0$,
\begin{equation}\label{re}
r^2_{k+1} = r^2_k -(r_{k+1} -r_k)^2- \frac{1}{\eta} \|\theta_{k+1}-\theta_k\|^2.
\end{equation}
This implies that $r_k$ is strictly decreasing and converges to $r^*$ as $k\to \infty$. Furthremore,  
\begin{equation}\label{rev}
\sum_{k=0}^\infty\|\theta_{k+1}-\theta_k\|^2 \leq \eta r^2_0.\quad\Longrightarrow\quad \lim_{k\to \infty} \|\theta_{k+1}-\theta_k\|=0.
\end{equation}
\end{theorem}
\begin{proof}
Starting from (\ref{aeng0}), we deduce  
$$
2r_{k+1}(r_{k+1}-r_k) = 2r_{k+1}v_k^\top(\theta_{k+1}- \theta_k)
=-\frac{1}{\eta}\|\theta_{k+1}-\theta_k\|^2.
$$
By rewriting with $2b(b-a) = b^2-a^2 + (b-a)^2$, we obtain equality (\ref{re}). This implies that $r^2_k$ is monotonically decreasing (also bounded below), ensuring  convergence. The non-negativity of $r_k$ further ensures the convergence of $r_k$. Summation of (\ref{re}) over $k$ from $0, 1, \cdots$ yields (\ref{rev}).  
\end{proof}


Note that, starting from (\ref{aeng0}b), the following relation is derived:  
$$
r_0 - r^* 
= 2\eta\sum_{j=0}^{\infty}r_{j+1}\|v_j\|^2 \Rightarrow \lim_{k\to \infty} (r_{k+1}\|v_k\|^2)=0. 
$$
In simpler terms,  after a finite number of steps, either $\|v_k\|^2$ becomes small or $r_{k}$ becomes small. For convergence to a stationary point in general objectives ($\|v_k\|^2 \to 0$), it is essential to control the rate of decrease of $r_k$. This control can be achieved by imposing  a moderate  upper bound on the base step size $\eta$.    

In Lemma \ref{tau}, we establish a sufficient condition on $\eta$ that ensures a positive lower bound for $(r_k)_{k=0}^\infty$,  a condition crucial  in the convergence proof of Theorem~\ref{cvg}. This is where the $\alpha$-smoothness assumption comes into play.  We define $L$ as $\alpha$-smooth if $\|\nabla^2 L(\theta)\|\leq \alpha$ for any $\theta \in \mathbb{R}^n$. Here,  $L^*:=\inf\limits_{\Theta} L$, and we assume $c$ is chosen such that $l^*:=\sqrt{L^*+c}>0$.  

\begin{lemma}\label{tau}
Assume $L$ is $\alpha$-smooth, bounded from below by $L^*$,  and $r_k$ generated by AEPG (\ref{aeng0}) with $T_k=A_k^{-1}$. If $r_0  \geq \frac{l(\theta_0)-l^*}{\lambda_1}$, then  
\begin{equation}\label{rge0}
r_k \geq r^* >0,  
\end{equation}
provided that:  
\begin{equation}\label{etas}
\eta < \eta_s: = \frac{4l^*\lambda_1}{\alpha r_0^2}\left(r_0-\frac{l(\theta_0)-l^*}{\lambda_1}\right) 
\end{equation}
where $\lambda_1>0$ is the smallest eigenvalue of $A_k$. 
\end{lemma}
\begin{proof}
Note that $l(\theta)=\sqrt{L(\theta)+c}$ is concave with respect to $L$,  and $\frac{dl}{dL}=\frac{1}{2l}$.  Thus, we have 
\begin{equation}\label{dl}
l(\theta_{j+1})-l(\theta_{j})\leq \frac{1}{2l(\theta_j)}(L(\theta_{j+1})-L(\theta_{j})).    
\end{equation}
Using the $\alpha$-smoothness of $L$, we have
\begin{equation}\label{l-smooth}
L(\theta_{j+1}) \leq L(\theta_j)+\nabla L(\theta_j)^\top (\theta_{j+1}-\theta_j) +\frac{\alpha}{2}\|\theta_{j+1}-\theta_j\|^2.   
\end{equation}
From (\ref{aeng0}a), we derive $\nabla L(\theta_j)= 2l(\theta_j)\nabla l(\theta_j)=2l(\theta_j) A_jv_j$. 
Hence,
\begin{align}\label{lmd}
\nabla L(\theta_j)^\top (\theta_{j+1}-\theta_j)\notag
& = 2l(\theta_j) (A_j v_j)^\top (-2\eta r_{j+1}v_j)\\\notag
& =-4\eta l(\theta_j) r_{j+1}\, v_j^\top A_j v_j \\
& \leq  -4\eta l(\theta_j) \lambda_1 r_{j+1}\|v_j\|^2 \\\notag
& = 2 l(\theta_j) \lambda_1(r_{j+1}-r_j),
\end{align}
where the last equality follows from a formulation of (\ref{aeng0}b). Substituting  (\ref{lmd}) and (\ref{l-smooth}) into (\ref{dl}), we obtain
$$
l(\theta_{j+1})-l(\theta_{j}) 
\leq \lambda_1(r_{j+1}-r_j) + \frac{\alpha}{4l(\theta_j)}\|\theta_{j+1}-\theta_j\|^2. 
$$
Taking a summation over $j$ from $0$ to $k-1$, we have: 
\begin{align*}
l(\theta_k)-l(\theta_{0}) &\leq \lambda_1(r_k-r_0) + \frac{\alpha}{4l^*}\sum_{j=0}^{k-1} \|\theta_{j+1}-\theta_j\|^2 \\
&\leq \lambda_1(r_k-r_0) + \frac{\alpha \eta }{4l^*}r^2_0.
\end{align*}
Considering $l(\theta_k)\geq l^*$ and the last inequality,  we obtain 
\begin{align*}
r_k & \geq r_0 -   \frac{l(\theta_0)-l^*}{\lambda_1} -\frac{\alpha \eta }{4l^*\lambda_1}r^2_0=\frac{\alpha r_0^2}{4l^*\lambda_1}(\eta_s-\eta)>0.
\end{align*}
This establishes a uniform lower bound for $r_k$. Letting $k\to\infty$, we obtain (\ref{rge0}).
\end{proof}

\begin{remark}\label{rmk:r_kb} Let us discuss the selection  of
$r_0=l(\theta_0)$, the default choice for consistency with the method derivation. When  $\lambda_1=1$, then  
$$
\eta_s= \frac{4}{\alpha}\frac{L^*+c}{L(\theta_0)+c} \sim \frac{4}{\alpha} 
$$
for $c$ large enough. This asymptotic limit clearly exceeds the upper threshold $2/\alpha$ which is known to be necessary for ensuring GD's stability.
\end{remark} 

\begin{remark}\label{rmk:r_kb+}
The stipulated lower bound for $r_0$ is not overly restrictive. In can be expressed equivalently as  
$$
\sqrt{L(\theta_0)+c}>\frac{\sqrt{L(\theta_0)+c}-l^*}{\lambda_1}. 
$$
This inequality holds for any value of $c$ satisfying $c +L^*>0$ when $\lambda_1 \geq 1$. However, for $\lambda_1 <1$, a larger $c$ becomes necessary to fulfill the condition:   
$$
c +L^* >\frac{(1-\lambda_1)^2}{\lambda_1(2-\lambda_1)}
(L(\theta_0)-L^*). 
$$
\end{remark} 

\begin{theorem}\label{cvg}
Under the same assumptions as in Lemma \ref{tau},  with $r_0 \geq 
\frac{l(\theta_0)-l^*}{\lambda_1}
$, let $(\theta_k)_{k=0}^\infty$ be generated by AEPG \eqref{aeng0}. Then $(L(\theta_k))_{k=0}^\infty$ monotonically
converges to a local minimum value of $L$ if
\begin{equation}\label{etau}
0<\eta\leq\min\bigg\{\eta_0, \eta_s\bigg\}, \quad 
\eta_0:=\frac{\lambda_1 l^*}{\alpha r_0}.
\end{equation}
\end{theorem}
\begin{proof} If $\eta \leq \eta_0$, the effective step size falls into the regime where $L(\theta_k)$ is shown to be decreasing, indicating convergence. With 
$r_0 \geq \frac{l(\theta_0)-l^*}{\lambda_1}$ and $\eta \leq \eta_s$, Lemma \ref{tau} ensures that $r_k$ is bounded below by a positive constant,  allowing  $\|\nabla L(\theta_k)\|^2$ to converge to zero as $k$ tends to infinity. For the detailed proof, please refer to Appendix \ref{pf-cvg} for the readers' convenience. 
\end{proof}

Let's discuss the convergence rate of AEPG with additional geometrical insights into  $L$, including properties such as convexity or the Polyak--Lojasiewicz (PL) condition. For a differentiable function $L:\R^n\to\R$ with $\arg\min L \neq\emptyset $ (indicating the optimization problem has at least one global minimizer), we say $L$ satisfies the PL condition if there exists $\mu>0$ such that:
\begin{equation}\label{PL}
\frac{1}{2}\|\nabla L(\theta)\|^2 \geq \mu (L(\theta)-L(\theta^*)),\quad\forall\theta\in\R^n,\quad \forall \theta^*\in \arg\min L.
\end{equation}
This condition implies that $\nabla L(\theta)=0$ implies either $L(\theta)=L(\theta^*)$ or $\theta \in \arg\min L$. In other words, critical points are global minimizers.

It's important to note that strongly convex functions ($\lambda_1(\nabla^2 f) \geq \mu$)  satisfy the PL condition (\ref{PL}). However, a function that satisfies the PL condition may not necessarily be convex. For instance, consider the function: 
\[
L(\theta)=\theta^2+3\sin^2\theta,\quad \theta \in \mathbb{R},
\]
which is nonconvex but satisfies the PL condition with $\mu=\frac{1}{32}$ and $\min L=0$.

\begin{theorem}\label{thm2}
Assume  $(\theta_k)_{k=0}^\infty \subset \Theta$ when $\Theta \neq \mathbb{R}^n$. The convergence rates of AEPG (\ref{aeng0}) are given in three distinct scenarios:\\
(i) For any $\eta>0$ and $r_0>0$, we have 
\begin{equation}\label{cgrad}
\min_{j<k}\|\nabla L(\theta_j)\|^2 \leq \frac{2r_0\lambda^2_n}{\eta r_k k}\left( 
\max_{j<k}L(\theta_j)+c
\right),  
\end{equation}
with $\lambda_n$ given in~\eqref{ass}. Under the assumptions of Theorem \ref{cvg} with $\eta$ satisfying \eqref{etau}, we have 
$r_k> r^*>0$, and the following convergence rates: \\ 
(ii) If $L$ is PL with a global minimizer $\theta^*$, then $\{\theta_k\}$ satisfies~\eqref{ctheta}, hence convergent:
\begin{align}\label{ctheta}
    & \sum_{k=0}^\infty\|\theta_{k+1}-\theta_k\|\leq \frac{4\lambda_n}{\sqrt{2\mu}\lambda_1} \sqrt{L(\theta_{0})-L(\theta^*)},\\\notag
    & L(\theta_k)-L(\theta^*)\leq e^{-c_0k r_k/\lambda_n}(L(\theta_{0})-L(\theta^*)), \quad c_0:=\frac{\mu \eta}{l(\theta_0)}.
\end{align}
(iii) If $L$ is convex with a minimizer $\theta^*$, then
\begin{equation}\label{cvg-cvx}
L(\theta_k)-L(\theta^*)\leq \frac{c_1\lambda_n\|\theta_0-\theta^*\|^2}{k r_k},\quad c_1=\frac{2l(\theta_0)}{\eta}.
\end{equation}
\end{theorem}
\begin{proof} The proof of (ii) and (iii) is somewhat standard and is therefore   deferred to Appendix \ref{pf2}. Here we present a proof for (i).    

(i) Using the scheme (\ref{aeng0}b), we have
$$
r_{j+1}-r_j 
= -2\eta r_{j+1}\|v_j\|^2.
$$
Take a summation over $j$ from $0$ to $k-1$ gives
$$
r_0 - r_k 
= 2\eta\sum_{j=0}^{k-1}r_{j+1}\|v_j\|^2
\geq 2\eta r_k \sum_{j=0}^{k-1}\|v_j\|^2
\geq \frac{2\eta r_k }{\lambda^2_n}\sum_{j=0}^{k-1}\|\nabla l(\theta_j)\|^2,
$$
which, upon rearrangement, leads to 
$$
k\min_{j<k}\|\nabla l(\theta_j)\|^2 \leq \sum_{j=0}^{k-1}\|\nabla l(\theta_j)\|^2 \leq \frac{r_0\lambda^2_n}{2\eta r_k}.
$$
Using $4l^2\|\nabla l\|^2=\|\nabla L\|^2$, and dividing both sides by $k$ gives result \eqref{cgrad}. 
\end{proof}

\begin{remark}\label{rmk:eta_k} 
\begin{enumerate}[leftmargin=*]
    \item Unlike GD, which uses a fixed step size constrained by the typically unknown smoothness constant $\alpha$ (thus making standard GD prone to non-convergence with larger step sizes), the gradient norm sequence's convergence rate to zero for AEPG in scenario (i) extends to general non-convex objective functions for any $\eta>0$. Notably, these observations remain valid regardless of the presence of $\alpha$.  It is natural to ponder  whether this convergence rate can be enhanced. Interestingly, the answer is negative, at least within the general class of functions under consideration  here.
    \item Regarding the upper bound in equation  (\ref{cgrad}), we can  establish the following inequality: 
\begin{equation}\label{Lbound}
L(\theta_k)\leq L(\theta_0)+\frac{\alpha \eta r_0^2}{2}.
\end{equation}
This bound holds when $L$ is $\alpha$-smooth. Using the $\alpha$-smoothness of $L$ and the descent direction of the search, we can express this as: 
\begin{equation*}
L(\theta_{j+1}) \leq L(\theta_j)+\nabla L(\theta_j)^\top (\theta_{j+1}-\theta_j) +\frac{\alpha}{2}\|\theta_{j+1}-\theta_j\|^2 \leq L(\theta_j)+\frac{\alpha}{2}\|\theta_{j+1}-\theta_j\|^2.   
\end{equation*}
By summing over $j$ from $0$ to $k-1$ and using equation (\ref{rev}), 
\begin{align*}
L(\theta_k)-L(\theta_{0}) &\leq  \frac{\alpha}{2}\sum_{j=0}^{k-1} \|\theta_{j+1}-\theta_j\|^2 
 \leq \frac{\alpha \eta }{2}r^2_0.  
\end{align*}
It is important to note that such a bound 
is unavailable for GD unless the step size is sufficiently small. 
\item It is crucial to highlight that the aforementioned results remain valid for a variable $\eta$,  as long as it adheres to the constraint specified in (\ref{etau}) and is not excessively small. The convergence theory with a variable $\eta$  provides the flexibility to adjust   $\eta$ when necessary.  For instance, the inclusion of the sequence  $(\theta_k)_{k=0}^\infty \subset \Theta$ becomes essential when $\Theta \neq \mathbb{R}^n$. This need for adaptability is addressed in Algorithm~\ref{alg2},  which we will introduce in Section~\ref{subsec:HRGD-linear}, where $\eta$ is selected based on a line search at every iteration. 
\end{enumerate}
\end{remark} 

\subsection{Equality constraints}\label{subsec:theory_constrained} 

In this section, we assume that
\begin{equation*}
    \Theta=\{\theta \in \mathbb{R}^n~|~B \theta=b\},
\end{equation*}
for some $B \in \mathbb{R}^{m \times n}$, $b\in \mathbb{R}^m$, and
\begin{equation}\label{eq:PG-1}
T_k=P(\theta_k) G(\theta_k)^{-1},
\end{equation}
where $G(\theta)$ is a symmetric positive definite matrix, and $P(\theta):\mathbb{R}^n \to \operatorname{ker}(B)$ is the $G(\theta)$-orthogonal projection operator on $\operatorname{ker}(B)$.

Before discussing convergence analysis results, we provide a motivation for the choice~\eqref{eq:PG-1}. To this end, assume that $\mathbb{R}^n$ is endowed with a Riemannian structure given by a metric tensor $G(\theta)$, $\theta \in \mathbb{R}^n$; that is, for all $\theta \in \mathbb{R}^n$ we have an inner product
\begin{equation}\label{eq:metricG}
\langle v_1,v_2 \rangle_{G(\theta)}=v_1^\top G(\theta) v_2,\quad v_1,v_2 \in T_\theta \mathbb{R}^n,
\end{equation}
where $T_\theta \mathbb{R}^n \cong \mathbb{R}^n$ is the tangent space of $\mathbb{R}^n$ at $\theta$. Note that
\[
T_\theta \Theta \cong \operatorname{ker} (B)=\{v \in \mathbb{R}^n~|~ B v =0\}.
\]
The choice~\eqref{eq:PG-1} is motivated by the following lemma.
\begin{lemma}\label{lem:proj_G}
For every smooth $f:\R^n \to \R$ one has that
\[
\operatorname{arg}\min_v \left\{ \frac{d f(\theta)}{dt}~\bigg|~\dot{\theta}=v,~v\in T_\theta \Theta,~\|v\|_{G(\theta)}\leq 1\right\}=-\frac{P(\theta)G(\theta)^{-1}\nabla f(\theta)}{\|P(\theta) G(\theta)^{-1}\nabla f(\theta)\|_{G(\theta)}}.
\]
Hence, the steepest descent direction of $f$ on the submanifold $(\Theta,G)$ is $-P(\theta) G(\theta)^{-1} \nabla f(\theta)$.
\end{lemma}
\begin{proof}
We have that
\begin{equation*}
    \begin{split}
        \frac{df(\theta)}{dt}=v^\top \nabla f(\theta) = v^\top G(\theta) G(\theta)^{-1} \nabla f(\theta)=\langle v, G(\theta)^{-1} \nabla f(\theta)\rangle_{G(\theta)}=\langle v, \nabla^G f(\theta)\rangle_{G(\theta)},
    \end{split}
\end{equation*}
where we denote by $\nabla^G f(\theta)=G(\theta)^{-1} \nabla f(\theta)$ the metric gradient. So we have the problem
\begin{equation*}
\begin{split}
    &\operatorname{arg}\min_v \left\{ \langle v, \nabla^G f(\theta)\rangle_{G(\theta)}~\bigg|~v\in T_\theta \Theta,~\|v\|_{G(\theta)}\leq 1\right\}\\
    =&\operatorname{arg}\min_v \left\{ \langle v, P(\theta)\nabla^G f(\theta)\rangle_{G(\theta)}~\bigg|~v\in T_\theta \Theta,~\|v\|_{G(\theta)}\leq 1\right\}\\
    =&-\frac{P(\theta)\nabla^G f(\theta)}{\|P(\theta)\nabla^G f(\theta)\|_{G(\theta)}}=-\frac{P(\theta)G(\theta)^{-1}\nabla f(\theta)}{\|P(\theta) G(\theta)^{-1}\nabla f(\theta)\|_{G(\theta)}},
\end{split}
\end{equation*}
where the first equation follows from the definition of the orthogonal projection.
\end{proof}

\begin{remark}
As previously mentioned, selecting  \eqref{eq:PG-1} in~\eqref{aeng0} yields $\theta_{k+1}-\theta_k =-2\eta r_{k+1} v_k \in T_{\theta_k} \Theta\cong \operatorname{ker}(B)$. Therefore,~\eqref{aeng0} returns feasible updates; that is,~\eqref{kin} is valid.
\end{remark}

The following theorem provides a concise summary of the stability and convergence properties of~\eqref{aeng0} in the context of the chosen ~\eqref{eq:PG-1}.
\begin{theorem}\label{thmT}
Let  $\theta_0 \in \Theta=\{\theta \in \mathbb{R}^n, \;B \theta =b\}$, and $G(\theta)$ be symmetric and positive definite with $0<\lambda_1 \leq \|G(\theta)\|_2 \leq \lambda_n$ for all $\theta \in \Theta$. Furthermore, let $(\theta_k)$ be generated by \eqref{aeng0} with~\eqref{eq:PG-1}.  The following statements hold: 
\begin{itemize}
\item[(1)] {\bf Unconditional Energy Stability:}  It satisfies unconditional energy stability, as stated in Theorem \ref{thm1}. 
\item[(2)]  {\bf Positive Lower Bound:} The statement in Lemma \ref{tau} regarding a positive lower bound for $r_k$ remains valid.  
\item[(3)]  {\bf Monotonic Convergence:}  Under the same assumptions and conditions on $\eta$ as described in Theorem \ref{cvg}, $L(\theta_k)$ decreases monotonically and converges towards a local minimum value of $L$ with
$$
\lim_{k\to \infty} \|P_k^\top \nabla L(\theta_k)\|\to 0.
$$
\item[(4)]  {\bf Convergence Rates:}  Convergence rates are obtained in three distinct scenarios:
\begin{itemize}
\item[(i)] For any $\eta>0$ and $r_0>0$, it holds that 
\begin{equation*}
\min_{j<k}\|P_j^\top \nabla L(\theta_j)\|^2 \leq \frac{2r_0\lambda^2_n}{\eta r_k k}\left(\max_{j<k}L(\theta_j)+c \right).
\end{equation*}
Under the same assumptions and conditions on $\eta$ as described in Theorem \ref{cvg}, where $r_k>r^*>0$,  the following convergence rates are established:  
\item[(ii)] If $L$ satisfies the projected PL condition with a minimum $\theta^* \in \Theta$:
\begin{equation}\label{PPL}
\frac{1}{2}\|P^\top\nabla L(\theta)\|^2 \geq \mu (L(\theta)-L(\theta^*)),\quad\forall \theta \in \Theta,
\end{equation}
where $\mu>0$, then the sequence $\{\theta_k\}$ has finite length, and hence converges. Furthermore, 
\begin{align}\label{ctheta+}    & L(\theta_k)-L(\theta^*)\leq e^{-c_0k r_k/\lambda_n}(L(\theta_{0})-L(\theta^*)), \quad c_0:=\frac{\mu \eta}{l(\theta_0)}.
\end{align}
\item[(iii)] If $L$ is convex with a minimum $\theta^* \in \Theta$, then:
\begin{equation}\label{cvg-cvx+}
L(\theta_k)-L(\theta^*)\leq \frac{c_1\lambda_n\|\theta_0-\theta^*\|^2}{k r_k},\quad c_1=\frac{2l(\theta_0)}{\eta}.
\end{equation}
\end{itemize}
\end{itemize}
\end{theorem}
\begin{proof}
(1) The proof for Theorem \ref{thm1} remains applicable in this context. 

(2) To demonstrate that (\ref{lmd}) still holds in the case of (\ref{eq:PG-1}), we recall that $P_j$ is the $G_j$-orthogonal projection operator, and $P_j$ is an involution, $P_j^2=P_j$, and $G_j$-symmetric, $G_j P_j= P_j^\top G_j$. Hence, $T_j=P_jG^{-1}_j=G^{-1}_jP^\top_j$ and $T_j = P_j T_j = P_j G^{-1}_jP^\top_j$.
We can express the gradient term as: 
\begin{align}\label{Fdtheta}
\nabla L(\theta_j)^\top\notag (\theta_{j+1}-\theta_j) 
&= -4\eta l(\theta_j) r_{j+1}\nabla l(\theta_j)^\top G^{-1}_jP_j^\top\nabla l(\theta_j) \\
&= -4\eta l(\theta_j) r_{j+1} \|P_j^\top\nabla l(\theta_j)\|^2_{G_j^{-1}}.
\end{align}
Considering that
$$
v_j=T_j\nabla l(\theta_j)=
G_j^{-1}P_j^\top\nabla l(\theta_j)\quad\Rightarrow\quad P_j^\top\nabla l(\theta_j)=G_jv_j,
$$
we further bound (\ref{Fdtheta}) by  
\begin{align*}
\nabla L(\theta_j)^\top (\theta_{j+1}-\theta_j) & =
-4\eta l(\theta_j) r_{j+1} \|G_jv_j\|^2_{G_j^{-1}} \\
& \leq -4 l(\theta_j) \eta\lambda_1 r_{j+1}\|v_j\|^2= 2 l(\theta_j) \lambda_1(r_{j+1}-r_j).
\end{align*} 
This establishes the validity of (\ref{lmd}), and the remaining portion of the proof aligns with that of Lemma \ref{tau}.

The proofs for (3) and (4) mirror those of Theorem \ref{thm2}, and therefore deferred to Appendix \ref{pf-thmT}. 
\end{proof}

\begin{remark}\label{rmk:PL} 
Rather than relying on the conventional gradient  $\nabla L(\theta)$, the convergence and convergence rate of $L(\theta_k)$  are influenced by the projected gradient
 $P^\top \nabla L(\theta)$. This is also evident at the continuous level: the projected PL condition (\ref{PPL}) implies:   
$$
\frac{d}{dt}L(\theta(t))= -\|P^\top \nabla L(\theta)\|^2_{G^{-1}} \leq -\frac{2\mu}{\lambda_n} 
(L(\theta(t))-L^*).
$$
Thus, for any $t>0$, 
$$
L(\theta(t))-L^* \leq e^{-2\mu t/\lambda_n} \left(L(\theta(0))-L^* \right),
$$
where $\theta(0) = \theta_0$, representing the initial guess.
\end{remark}

\begin{remark} In general, verifying the projected PL condition directly can be challenging. To illustrate this new structural condition, consider an example involving loss functions of the form:
$$
L(\theta)=\frac{1}{2}\left(\beta \theta_1^2+\alpha \theta_2^2 \right), 
$$
subject to a general linear constraint 
$$
q(\theta)=a\theta_1+b\theta_2 -1=0,  
$$
where $ab\not=0, \alpha \geq \beta >0$ are constants. This constrained minimization problem is convex, thereby admiting  a unique solution:  
$$
\theta^* = \bigg(\frac{a\alpha}{a^2\alpha+b^2\beta}, \frac{b\beta}{a^2\alpha+b^2\beta}\bigg). 
$$
Upon a careful calculation (refer to Appendix \ref{pplfunc} for details), it can be verified that the projected PL condition holds for any $\theta \in \Theta$, where 
$$
\mu = \frac{a^2\alpha+b^2\beta}{a^2+b^2}.
$$
\end{remark}

To apply the method and theoretical results to a specific  optimization task, it remains essential to identify and compute matrices $G_j^{-1}$. This aspect will be addressed in the subsequent sections.  

\begin{remark}
It is observed that the convergence rate for convex objectives in (4)-(iii) remains the same as that in the scenario $T_k=A_k^{-1}$.
\end{remark} %

\section{Hessian-Riemannian metric}\label{ap}

In Section~\ref{subsec:theory_constrained}, we explored  preconditioning matrices of the form~\eqref{eq:PG-1}, where $G$ is a generic metric. In this section, we  discuss a more specific choice for $G$ based on the Hessian of a suitable convex function. The preconditioned gradient descent algorithm~\eqref{eq:pre-gd} based on such $G$ is known as the Hessian-Riemannian gradient descent~\cite{ABB04}.

For analysis purposes, we previously assumed only affine equality constraints. In this section, we do not perform analysis and include more general convex inequality constraints. We believe the reader will benefit from this more general exposition, and it will motivate our future work. More precisely, we introduce
\begin{equation}\label{eq:M=suplevelU}
    \M=\{\theta~|~U(\theta)\geq 0\},
\end{equation}
where $U$ is a concave function, and assume
\begin{equation}\label{eq:ThetawithU}
\Theta=\left\{\theta~|~U(\theta)\geq0,~B\theta=b \right\}.
\end{equation}
Thus,~\eqref{minp} reduces to
\begin{align*}
\min\quad  & L(\theta) \\
\text{s.t.}\quad  & B \theta=b,\quad \text{(linear equality constraints)} \\
& U(\theta)\geq 0. \quad \text{(inequality constraints)}
\end{align*}
In large-scale nonlinear programming, popular methods for solving the above problem include the interior-point method and active-set SQP methods \cite{NW06}. 
In this section, we show that this problem can be more efficiently solved using AEPG~\eqref{aeng0} with a suitable choice of $T_k$.

\subsection{Hessian-Riemannian formulation} 

We start by reviewing the Hessian-Riemannian framework in the geometric context preceding Lemma~\ref{lem:proj_G}. We first consider only inequality constraints; that is,
\begin{equation}\label{mp}
\min\{ L(\theta)\;|\; \theta \in \M \}, 
\end{equation}
where $L\in C^1$ is bounded from below,  and $\M \subset \R^n$ is a closed convex set such that $\operatorname{int}(\M)\neq \emptyset$. 
Standard gradient descent for~\eqref{mp} may not necessarily stay in $\M$. To address this limitation, one approach is to introduce a suitable Riemannian metric on $\M$ that shrinks the gradients near $\partial \M$ to prevent updates from leaving $\M$.

Assume $\M$ is endowed with a metric $G$ as in~\eqref{eq:metricG}. Then for a smooth $f:\M \to \R$ the metric gradient and chain rule are
\[
\nabla^G f(\theta)=G(\theta)^{-1}\nabla f(\theta),\quad \frac{df(\theta)}{dt}=\left\langle \nabla^G f(\theta),\dot{\theta} \right\rangle_{G(\theta)},
\]
where $t\mapsto \theta(t) \in \mathcal{M}$ is an arbitrary smooth curve (see Lemma~\ref{lem:proj_G}).

For a convex function $L$,  the variational characterization of $\theta^*\in \arg\min\limits_{\M} L$ is given by: 
\[
\langle \nabla^G L(\theta), \theta-\theta^*\rangle_{G(\theta)} \geq 0, \quad \forall \theta \in \M.
\]
Therefore,  if there exists a function $V$ such that $\nabla^G V(\theta)=\theta-\theta^*$, then $V$ serves as a natural Lyapunov function for the gradient flow
\[
\dot{\theta}=-\nabla^G L(\theta).
\]
Indeed, it follows that
\[
\frac{dV(\theta)}{dt}=\langle \nabla^G V(\theta),\dot{\theta}\rangle_{G(\theta)}=-\langle \nabla^G L(\theta), \theta-\theta^*\rangle_{G(\theta)} \leq 0.
\]
The existence of such Lyapunov functions is valuable for proving convergence results for gradient descent algorithms. The following theorem~\cite{ABB04} characterizes $G$ for which such $V$ can be found. Furthermore, these $V$ are nothing else but Bregman divergences.

\begin{theorem}[Theorem 3.1~\cite{ABB04}]\label{thm-H}
A metric $G \in C^1(\operatorname{int}(\M))$ ensures that for any given $\xi \in \operatorname{int}(\M)$, there exists a functional $V_{\xi}:\operatorname{int}(\M)\to\R$ satisfying $\nabla^G V_{\xi}(\theta)=\theta-\xi$ if and only if there exists a strictly convex function $h\in C^3(\operatorname{int}(\M))$ such that $\forall\theta\in\operatorname{int}(\M)$, $G(\theta)=\nabla^2 h(\theta)$. Additionally, defining $D_h:\operatorname{int}(\M)\times\operatorname{int}(\M)\to\R$ by
\begin{equation}\label{bmd}
D_h(\xi,\theta)=h(\xi)-h(\theta)-\langle\nabla h(\theta),\xi-\theta\rangle
\end{equation}
and taking $V_\xi=D_h(\xi,\cdot)$, we get 
$\nabla^G V_\xi(\theta)=\theta-\xi$. 
\end{theorem}
The preceding discussion motivates the application of the metric
\begin{equation}\label{eq:G=Hessh}
    G(\theta)=\nabla^2 h(\theta),\quad \theta \in \operatorname{int}(\M),
\end{equation}
for a suitable convex function $h:\operatorname{int}(\M) \to \R$. There are many choices of $h$ for a fixed $\M$. The \textit{Legendre type}  identified in \cite{ABB04}, is a class of $h$ suitable for enforcing $\M$ and performing convergence analysis. Precisely, we say that $h$ is of Legendre type if: 
\begin{enumerate}
\item $h\in C^2(\operatorname{int}(\M))$ is strictly convex,
\item $|\nabla h(\xi_j)|\to +\infty$ for all $\{\xi_j\}\subset\operatorname{int}(\M)$ converging to a boundary point of $\M$.
\end{enumerate}
For constraints described by $U$ as in ~\eqref{eq:M=suplevelU}, we discuss the construction of $h$ in Appendix \ref{cons}.

\subsection{Adding linear equality constraints}\label{subsec:HRGD-linear}

We now consider the case when~\eqref{eq:ThetawithU} holds; that is, we have the optimization problem:
\begin{equation}\label{mp+}
\min \{ L(\theta)\;|\; U(\theta)\geq 0,~ B\theta=b\}.
\end{equation}
Following Section~\ref{subsec:theory_constrained}, we can combine the Hessian-Riemannian metric and the projection operator onto $\operatorname{ker}(B)$. Thus,  \eqref{mp+} is in the form of \eqref{minp} and can be solved by \eqref{aeng0} with 
\begin{equation}\label{H}
\begin{split}
    T_k=P(\theta_k) \nabla^2 h(\theta_k)^{-1}=\nabla^2 h(\theta_k)^{-1} \left( I- B^\top (B \nabla^2 h(\theta_k)^{-1}B^\top)^{-1}B 
\nabla^2 h(\theta_k)^{-1} \right).
\end{split}
\end{equation}

\begin{algorithm}
\caption{AEPG for solving problem~\eqref{mp+} with $U$ in the form of~\eqref{dm}.  
}
\label{alg2}
\begin{algorithmic}[1] 
\Require $G$: a metric in $\M$ and $G(\theta)=\nabla^2 h(\theta)$ for some $h$ of the Legendre type; $B$: equality constraint matrix; $c$: a parameter such that $L({\theta})+c>0$ for all $\theta \in \M$; 
$\eta^*$: an upper bound for the step size; 
$\epsilon\in[0,\frac{1}{2}]:$ a small positive constant; and
$K$: the total number of iterations.
\Require ${\theta}_0$: initial guess of $\theta$ satisfying $B\theta_0=b$ and $U_i(\theta_0)\geq0$ for $i\in[p]$; ${r}_{0}=\sqrt{L({\theta}_0)+c}$: initial energy.
\For{$k=0$ to $K-1$}
\State Compute: $\nabla^2 h(\theta_k)^{-1}$
\State $\nabla l(\theta_{k}) = \nabla L(\theta_{k}) /(2\sqrt{L({\theta}_{k})+c})$ 
\State $A_k^\dagger=\nabla^2 h(\theta_k)^{-1} \left( I- B^\top (B \nabla^2 h(\theta_k)^{-1}B^\top)^{-1}B 
\nabla^2 h(\theta_k)^{-1} \right)$
\State $v_k = A_k^\dagger\nabla l(\theta_{k})\quad$ (compute Riemannian gradient)
\State Line search: $\eta_k = {\rm clip}\big({\rm argmax}\{\eta\;|\;I_i(\eta)\geq\epsilon U_i(\theta_k),i\in[p]\},0,\eta^*\big)$ 
\State $r_{k+1} = r_{k}/(1+2\eta_k\|v_k\|^2)\quad$ (update energy)
\State ${\theta}_{k+1} = {\theta}_{k} - 2\eta_k r_{k+1}v_k\quad$ (update parameters)
\EndFor
\State \textbf{return} ${\theta}_K$
\end{algorithmic}
\end{algorithm}

\begin{remark}
The line search step in Algorithm~\ref{alg2} is to ensure that $(\theta_k)_{k\geq0}$ stay in $\operatorname{int}(\M)$. 
Simultaneously, our experiments show that this can be simply guaranteed by choosing $\eta$ suitably small.
\end{remark}

%
\section{Natural gradient descent}\label{sec:NGD}

In this section, we show that the HRGD can be cast as an NGD.

\subsection{Natural Gradient Descent}\label{subsec:NGD}

We present the Natural Gradient Descent (NGD) following the exposition in~\cite{NLY22}. Let $(\mathcal{M},g)$ be a (formal) Riemannian manifold, $\Theta \subset \mathbb{R}^k$ a closure of a non-empty open set, $\phi:\Theta \mapsto \mathcal{M}$ a smooth forward model that parametrizes $\M$ explicitly, and $f:\mathcal{M}\to \mathbb{R}$ a smooth function. Furthermore, consider the optimization problem
\begin{equation}\label{eq:L_composite}
    \min_{\theta \in \Theta} f(\phi(\theta)).
\end{equation}
Let $\left\{\partial^g_{\theta_i} \phi(\theta) \right\}_{i=1}^k \subset T_{\phi(\theta)}\mathcal{M}$ be the tangent vectors. Then the NGD direction for this problem is given by
\begin{equation}\label{eq:p_nat}
    p^{nat}=-G(\theta)^{-1} \nabla_\theta f(\phi(\theta)), 
\end{equation}
where
\begin{equation*}
    G_{ij}(\theta)=\langle \partial^g_{\theta_i}\phi(\theta), \partial^g_{\theta_j} \phi(\theta) \rangle_{g(\phi(\theta))},\quad 1\leq i,j \leq k,~\theta \in \Theta,
\end{equation*}
is called an \textit{information matrix}. This choice of the metric corresponds to the steepest descent as measured in the ``natural metric'' of the model-manifold $(\mathcal{M},g)$. Hence, there is an inherent robustness with respect to the parameterization $\theta \mapsto \phi(\theta)$~\cite{NLY22}.

For simplicity, we assume that $G(\theta)$ is invertible for all $\theta \in \Theta$. Otherwise, $G(\theta)^{-1}$ should be replaced by the pseudoinverse $G(\theta)^\dagger$. 

To establish a connection between NGD and Hessian-Riemannian Gradient Descent (HRGD), we first present a variational formulation of $p^{nat}$ in~\eqref{eq:p_nat}. The following lemma is elementary and can be found in many works on NGD. Nevertheless, we present it here for the convenience of the reader. For a more comprehensive discussion of the subject,  we refer to~\cite{NLY22} and numerous references therein.

\begin{lemma}\label{lma:p_nat_variational}
    Let $p^{nat}$ be given by~\eqref{eq:p_nat}. Then one has that 
    \begin{equation}\label{eq:p_nat_var}
        p^{nat}=\arg\min_{p\in \mathbb{R}^k} \left\|\nabla^g f(\phi(\theta))+\sum_{i=1}^k p_i \partial^g_{\theta_i} \phi(\theta) \right\|_{g(\phi(\theta))}^2.
    \end{equation}
\end{lemma}
\begin{proof}
    Expanding the square norm,  we have: 
    \begin{equation}\label{eq:sq_norm_expansion}
        \begin{split}
            &\left\|\nabla^g f(\phi(\theta))+\sum_{i=1}^k p_i \partial^g_{\theta_i} \phi(\theta) \right\|_{g(\phi(\theta))}^2=p^\top G(\theta) p+2 p^\top \nabla_\theta f(\phi(\theta))+\left\|\nabla^g f(\phi(\theta))\right\|_{g(\phi(\theta))}^2,\\
        \end{split}
    \end{equation}
    where we used the chain rule
    \begin{equation*}
        \langle \nabla^g f(\phi(\theta)), \partial^g_{\theta_i} \phi(\theta)\rangle_{g(\phi(\theta))}=\partial_{\theta_i} f(\phi(\theta)),\quad 1\leq i \leq k.
    \end{equation*}
    Hence, the first-order optimality condition with respect to $p$ in~\eqref{eq:sq_norm_expansion} yields~\eqref{eq:p_nat_var}.
\end{proof}

\subsection{Hessian-Riemannian Gradient Descent as a Natural Gradeint Descent}

Consider again the constrained optimization problem
\begin{equation}\label{eq:xminp}
\min_{x\in \mathbb{R}^d}\quad  \left\{ f(x)~:~U(x)\geq 0,~Bx=b\right\},
\end{equation}
where $U$ is a concave function, and $B\in \mathbb{R}^{m \times d}$ such that $\operatorname{rank}(B)=m<d$. As before, we denote by
\begin{equation}\label{eq:xM}
    \begin{split}
        \mathcal{M}=\{x~:~U(x)\geq 0\},
    \end{split}
\end{equation}
and assume that $h:\operatorname{int}(\mathcal{M}) \to \mathbb{R}$ is a convex function of Legendre type. Recall that the HRGD direction is given by
\begin{equation}\label{eq:xHRGDdir}
    \dot{x}=-P(x) \nabla^2 h(x)^{-1} \nabla f(x),
\end{equation}
where $P(x):\mathbb{R}^d \to \operatorname{ker}(B)$ is the $\nabla^2 h(x)$-orthogonal projection. Our goal is to show that~\eqref{eq:xHRGDdir} can be interpreted as an NGD direction.

Since $\operatorname{rank}B=m$, the solutions of $Bx=b$ have a parametric representation by an affine map $\phi:\mathbb{R}^{d-m} \to \mathbb{R}^d$; that is,
\begin{equation*}
    \left\{x~:~Bx=b\right\} =\left\{\phi(\theta)~:~\theta \in \mathbb{R}^{n}, \quad n:=d-m\right\}.
\end{equation*}
Moreover, without loss of generality, we can assume that the Jacobian of $\phi$ has the form
\begin{equation*}
    D_\theta \phi(\theta)=\begin{pmatrix}W\\I \end{pmatrix},
\end{equation*}
where
\begin{itemize}
    \item $W\in \mathbb{R}^{m \times n}$,
    \item $I\in \mathbb{R}^{n \times n}$ is the identity matrix,
    \item $\operatorname{rank}D_\theta \phi(\theta)=n$,
    \item the column vectors of $D_\theta \phi(\theta)$ form a basis for $\operatorname{ker}(B)$.
\end{itemize}

Hence, denoting by $\Theta=\phi^{-1}(\mathcal{M})$, we obtain that~\eqref{eq:xminp} can be written as
\begin{equation*}
    \min_{\theta \in \Theta} f(\phi(\theta)),
\end{equation*}
recovering the setup in~\eqref{eq:L_composite}.
\begin{lemma}
    Consider the problem~\eqref{eq:xminp}. Suppose that $\Theta,\phi$ are given as above, and $\mathcal{M}$ is equipped with the Hessian metric
    \begin{equation*}
        \langle v_1,v_2 \rangle_{g(x)}=v_1 \cdot \nabla^2 h(x) v_2,\quad v_1,v_2 \in T_x \mathcal{M}\cong \mathbb{R}^d.
    \end{equation*}
    Furthermore, assume that $\dot{\theta}=p^{nat}$, where $p^{nat}$ is defined as in~\eqref{eq:p_nat}. Then for $x=\phi(\theta)$ we have that the equality~\eqref{eq:xHRGDdir} is valid.
\end{lemma}
\begin{proof}
    From Lemma~\ref{lma:p_nat_variational} we have that
    \begin{equation*}
        p^{nat}=\underset{p\in \mathbb{R}^{n-m}}{\operatorname{argmin}} \left\| \nabla^g f(\phi(\theta))+D_\theta \phi(\theta)~p\right\|^2_{\nabla^2 h(\phi(\theta))},
    \end{equation*}
    Since the columns of $D_\theta \phi(\theta)$ form a basis in $\operatorname{ker}(B)$, we have that vectors of the form $D_\theta \phi(\theta) p$ span the whole subspace $\operatorname{ker}(B)$, and so
    \begin{equation*}
        D_\theta \phi(\theta) p^{nat}=-P(\phi(\theta)) \nabla^g f(\phi(\theta))=-P(\phi(\theta)) \nabla^2 h(\phi(\theta))^{-1} \nabla f(\phi(\theta)),
    \end{equation*}
    where $P(\phi(\theta))$ is the $\nabla^2 h(\phi(\theta))$-orthogonal projection on $\operatorname{ker}(B)$, and we used the fact that
    \begin{equation*}
    \nabla^g f(x)=\nabla^2 h(x)^{-1} \nabla f(x),\quad \forall x \in \operatorname{int}(\mathcal{M}).     
    \end{equation*}
    
    Hence, using the Chain Rule, we obtain
    \begin{equation*}
        \dot{x}=D_\theta \phi(\theta) \dot{\theta}=D_\theta \phi(\theta) p^{nat}=-P(\phi(\theta)) \nabla^2 h(\theta)^{-1} \nabla f(\phi(\theta))=-P(x)\nabla^2 h(x)^{-1} \nabla f(x).
    \end{equation*}
\end{proof}

 %

\section{Wasserstein metric}\label{Wasser}

In Section~\ref{sec:NGD}, we explored the NGD in the context of a general (formal) Riemannian manifold $(\mathcal{M},g)$. In this section, our focus shifts to the Riemannian manifold induced by the Wasserstein metric, commonly known as the optimal transportation metric. This metric has gained recent popularity in data science and inverse problem communities, which explains our motivation to pay special attention to the Wasserstein metric and discuss the computational aspects of the corresponding NGD and AEPG algorithms.

In particular, we present how to compute tangent vectors $\{\partial_{\theta_i}^g \phi(\theta)\}_{i=1}^k$ and discuss efficient methods of computing the NGD direction $p^{nat}$ in~\eqref{eq:p_nat} following the discussion in~\cite{NLY22}. Subsequently,  we combine the Wasserstein NGD with the AEPG algorithm~\eqref{aeng0} and obtain an adaptive Wasserstein NGD algorithm described in Algorithm~\ref{alg3}.

Let $\mathcal{M}=\mathcal{P}_{2,ac}(\mathbb{R}^d)$ be the set of Borel probability measures in $\mathbb{R}^d$ with finite second moments that are absolutely continuous with respect to the Lebesgue measure in $\mathbb{R}^d$. In what follows, we slightly abuse notation, using same symbols for both probability measures and their density functions.

The quadratic Wasserstein distance is then defined as
\begin{equation}\label{eq:W2}
\begin{split}
    &W_2(\rho_1,\rho_2)=\inf_{\pi \in \mathcal{P}_2(\mathbb{R}^{2d})}\left(\int_{\mathbb{R}^{2d}} |x-y|^2 d\pi(x,y) \right)^{\frac{1}{2}}\\
    \text{s.t.}~&\int_{\mathbb{R}^{2d}} \phi(x) d\pi(x,y)=\int_{\mathbb{R}^{d}} \phi(x) d\rho_1(x),\quad \forall \phi \in C^\infty_c(\mathbb{R}^d),\\
    &\int_{\mathbb{R}^{2d}} \psi(y) d\pi(x,y)=\int_{\mathbb{R}^{d}} \psi(y) d\rho_2(y),\quad \forall \psi \in C^\infty_c(\mathbb{R}^d),
\end{split}
\end{equation}
for all $\rho_1,\rho_2 \in \mathcal{M}$. It turns out that $W_2$ can be interpreted as a geodesic distance on a (formal) Rimennian manifold as follows ~\cite{AMS08}. For $\rho \in \mathcal{M}$ we set
\begin{equation}\label{eq:tangent_W2}
    T_\rho \mathcal{M}= \overline{\left\{ \nabla \phi~:~\phi\in C^\infty_c(\mathbb{R}^d)\right\}}^{L^2_{\rho}(\mathbb{R}^d;\mathbb{R}^d)},
\end{equation}
and
\begin{equation}\label{eq:inner_W2}
    \langle v_1,v_2 \rangle_{g(\rho)}=\int_{\mathbb{R}^d} v_1(x) \cdot v_2(x) \rho(x)dx,\quad \forall v_1,v_2 \in T_\rho \mathcal{M}. 
\end{equation}

Our goal is to solve the problem
\begin{equation}\label{prho+}
\min_{\theta\in\Theta} L(\theta):=f(\rho(\theta, \cdot)),
\end{equation}
where $\Theta \subset \mathbb{R}^n$ is a closure of a non-empty open set,  $f:\mathcal{M} \to \mathbb{R}$, and $\rho(\theta,\cdot) \in \mathcal{M}$ for all $\theta \in \Theta$.

In this setting, the NGD direction of $L$ is given by
\begin{equation}\label{eq:p_nat_W2}
p^{W}=-G(\theta)^{-1}\nabla L(\theta),     
\end{equation}
where $G(\theta)\in\mathbb{R}^{n\times n}$ is the information matrix
\begin{equation}\label{eq:G}
    G_{ij}(\theta)=\int_{\R^d}\partial^W_{\theta_i}\rho(\theta,x)\cdot \partial^W_{\theta_j}\rho(\theta,x) ~ \rho(\theta,x) dx  ,\quad 1\leq i,j \leq n, 
\end{equation}
and $\{\partial^W_{\theta_i}\rho\} \subset T_\rho \mathcal{M}$ are the suitable tangent vectors.

Lemma~\ref{lma:p_nat_variational} yields that
\begin{equation}\label{lsp2}
p^W = \argmin_{p \in\R^n}\bigg\|\partial^W_\rho f + \sum_{i=1}^{n}p_i\partial^W_{\theta_i}\rho\bigg\|^2_{L^2_{\rho}(\R^d;\R^d)},   
\end{equation}
where $\partial^W_\rho f $ is the Wasserstein gradient of $f$ at $\rho$.

\begin{proposition}[Proposition 2.2 in \cite{NLY22}]\label{pps}
Let $\partial_\rho f$ and $\{\partial_{\theta_i}\rho\}_{i=1}^n$ be, respectively, the $L^2$ derivative and tangent vectors; that is, the derivative and tangent vectors in the standard sense of calculus of variations. Then we have that
\begin{align}\label{Wf}
\partial^W_\rho f &= \nabla \partial_\rho f,\\ \label{min}
\partial^W_{\theta_i}\rho & = \argmin_v \left\{\|v\|^2_{L^2_{\rho}(\R^d;\R^d)}\;|\;-\nabla \cdot(\rho v)=\partial_{\theta_i}\rho \right\},\quad i=1,...,n.
\end{align}
\end{proposition}
An interesting fact is that the minimization problem (\ref{min}) can be characterized by using a potential function~\cite[Sections 8.1.2 and 8.2]{villani2003topics}, \cite[Section 4]{mallasto2019formalization}, \cite[Section 3]{chen2020optimal}, \cite[Section 2]{LM18}.  
\begin{lemma}
One has that
\begin{align}\label{minv}
&\min_v\{\|v\|^2_{L^2_{\rho}(\R^d;\R^d)}\;|\;-\nabla \cdot(\rho v)=\partial_{\theta_i}\rho\}\\\label{minphi}
= &\min_\phi \left\{\|\nabla\phi(x)\|^2_{L^2_{\rho}(\R^d;\R^d)}\;|\; \int_{\R^d} \phi(x)dx = 0, -\nabla \cdot(\rho \nabla\phi(x))=\partial_{\theta_i}\rho \right\}.
\end{align} 
This minimization problem thus admits the following solution:
\begin{equation}   
\label{Wrho}
\partial^W_{\theta_i}\rho = (\nabla (-\Delta_\rho)^{-1})\partial_{\theta_i}\rho,\quad i=1,...,n.    
\end{equation}
\end{lemma} 

\begin{remark}
Note that (\ref{Wrho}) is well-defined at the continuous level, while for computational efficiency, we still use (\ref{min}).    
\end{remark}

By Proposition \ref{pps}, given the $L^2$ tangent vectors $\{\partial_{\theta_1}\rho, \cdots, \partial_{\theta_n}\rho\}$ and gradient $\partial_\rho f$, the Wasserstein natural gradient can be calculated in two steps:
\begin{enumerate}
\item Compute $\tilde v_i =\sqrt{\rho} \partial^W_{\theta_i}\rho$ for $i=1,...,n$ by
\begin{equation}\label{rhov}
\tilde v_i = \argmin_{\tilde v}\{\|\tilde v\|^2_{L^2(\R^d;\R^d)}\;|\; \mathbf{M}\tilde v=\partial_{\theta_i}\rho\},\quad\text{where}\quad \mathbf{M}\tilde v=-\nabla \cdot(\sqrt{\rho}\tilde  v).
\end{equation}
Here $\tilde v_i$ 
is uniquely defined by $\partial_{\theta_i}\rho$, denoted by $\mathbf{M}^\dagger (\partial_{\theta_i}\rho) $. 
\item Compute the Wasserstein natural gradient by
\begin{equation}\label{rhof}
p^W=\argmin_{p \in \R^n} \bigg\|\sqrt{\rho}\nabla \partial_\rho f +\sum_{i=1}^n p_i \tilde v_i \bigg\|^2_{L^2(\R^d;\R^d)}.
\end{equation}
\end{enumerate}

Upon further spatial discretization, $p^W$ can be conveniently used for updating $\theta_k$ in AEPG algorithm stated below. For further details about the spatial discretization, we refer to \cite{NLY22}.

\begin{algorithm}
\caption{AEPG for solving the problem (\ref{prho+})  
}
\label{alg3}
\begin{algorithmic}[1] 
\Require  $\rho(\theta)$, $f$ a loss function; $c$: a constant such that $L({\theta})+c>0$, where $L(\theta)=f(\rho(\theta))$; 
$\eta$: base step size; and 
$T$: the total number of iterations.
\Require ${\theta}_0$: initial guess of $\theta$; ${r}_{0}=l(\theta_0)=\sqrt{L({\theta}_0)+c}$: initial energy.
\For{$k=0$ to $T-1$}
\State compute $p^W_k$ via~\eqref{rhov} and~\eqref{rhof} (update natural gradient)
\State $v_k = -p^W_k/2l(\theta_k)$
\State $r_{k+1} = r_{k}/(1+2\eta \|v_k\|^2)\quad$ (update energy)
\State ${\theta}_{k+1} = {\theta}_{k} - 2\eta r_{k+1}v_k\quad$ (update parameters)
\EndFor
\State \textbf{return} ${\theta}_T$
\end{algorithmic}
\end{algorithm}

%
\section{Numerical examples}\label{numeric}

This section presents a series of optimization examples to illustrate \footnote{The code is available at \url{https://github.com/txping/AEPG}.}
\begin{enumerate}
\item The advantages of natural gradient over the standard gradient.
\item The enhanced convergence of AEPG over HRGD (\ref{eq:pre-gd}) with $T_k$ given by (\ref{H})) and WNGD ((\ref{eq:pre-gd}) with $T_k=G(\theta_k)^{-1}$,  where $G(\theta)$ is the information matrix given by (\ref{eq:G})), particularly in addressing ill-conditioned or nonconvex problems.
\end{enumerate}

In Subsection \ref{hes}, we first present benchmark convex and nonconvex constrained optimization problems in the form of \eqref{mp}. These problems are solved by HRGD by constructing a Hessian matrix $\nabla^2 h$ dictated by the form of constraints and then applying the AEPG method. We show the advantage of AEPG over HRGD,  
especially in handling ill-conditioned or nonconvex problems. Furthermore, we apply AEPG to address the D-optimal design problem,  showcasing that with the preconditioning matrix identified by the Hessian--Riemannian metric, AEPG exhibits advantages in both efficiency and accuracy.

In Subsection~\eqref{ex-wng}, we delve into an optimization problem on the Wasserstein Riemannian manifold presented in the form of~\eqref{eq:L_composite}. We employ the least-squares formulation (\ref{rhov}) and (\ref{rhof}) to efficiently compute the Wasserstein natural gradient. 
Our results indicate that methods utilizing the standard gradient (GD and AEGD) may get stuck at a local minimum, whereas methods employing the Wasserstein natural gradient (WNGD and AEPG) reliably converge to the global minimum.

Throughout all experiments, we fine-tune the step size of each method to ensure they solve the problem with the minimum number of iterations or the least computational time. 

\subsection{Hessian-Riemannian method}\label{hes}
In the first two examples, we assess the performance of HRGD and AEPG on functions with varying condition numbers (specifically, the condition number of $\nabla^2 f$). More precisely, we set the stopping criterion as $|f(x)-f^*|<\epsilon$ and compare the number of iterations each algorithm takes to achieve the specified accuracy. Additionally, we calculate the ratio of HRGD iterations to AEPG iterations. The summarized results are presented in Tables \ref{tb:quad} and \ref{tb:rosen}. Both sets of results indicate that AEPG significantly enhances the convergence of HRGD, particularly in the context of ill-conditioned and nonconvex problems.

\begin{figure}[ht]
\begin{subfigure}[b]{0.5\linewidth}
\centering
\includegraphics[width=1\linewidth]{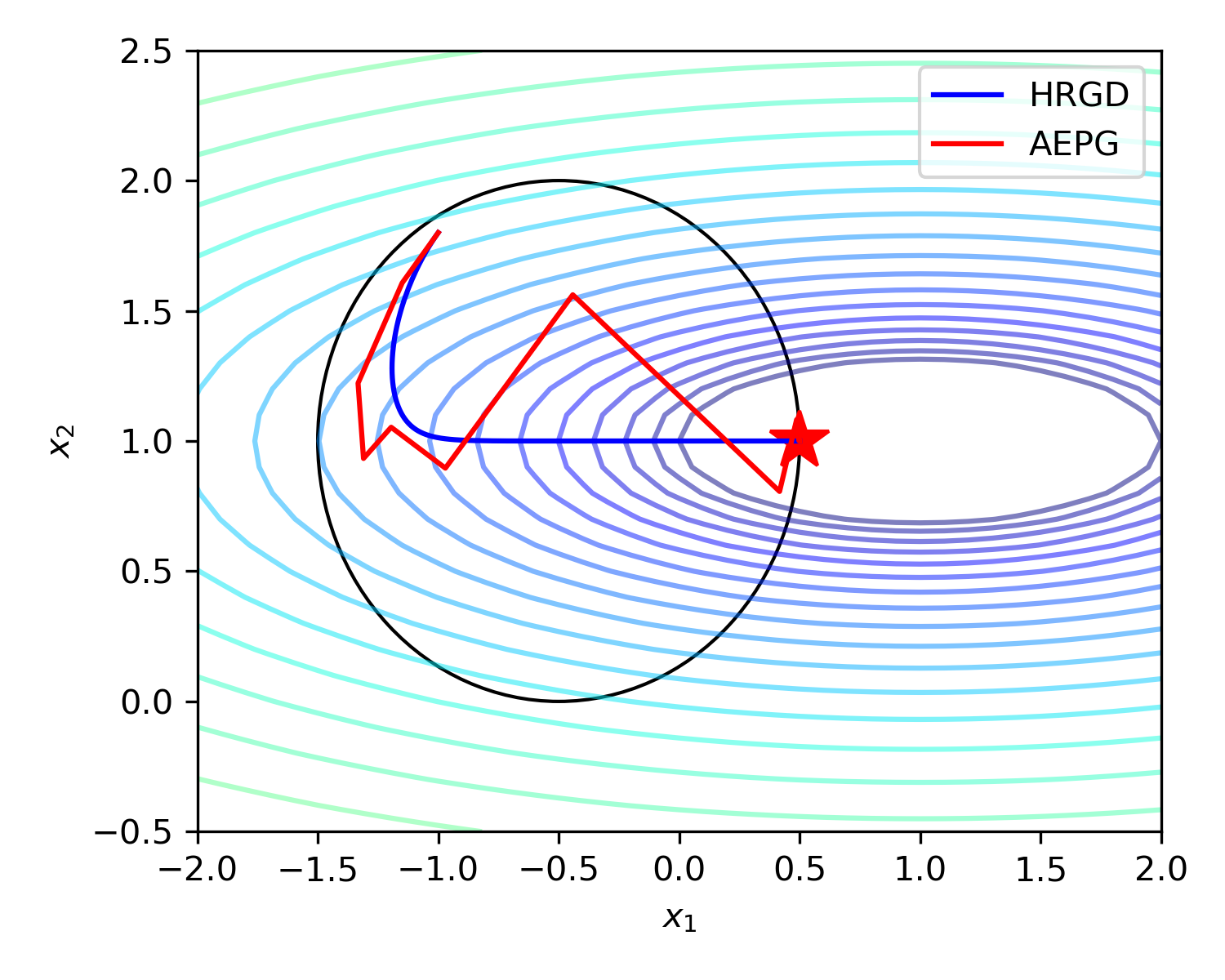}
\caption{Quadratic}
\end{subfigure}%
\begin{subfigure}[b]{0.5\linewidth}
\centering
\includegraphics[width=1\linewidth]{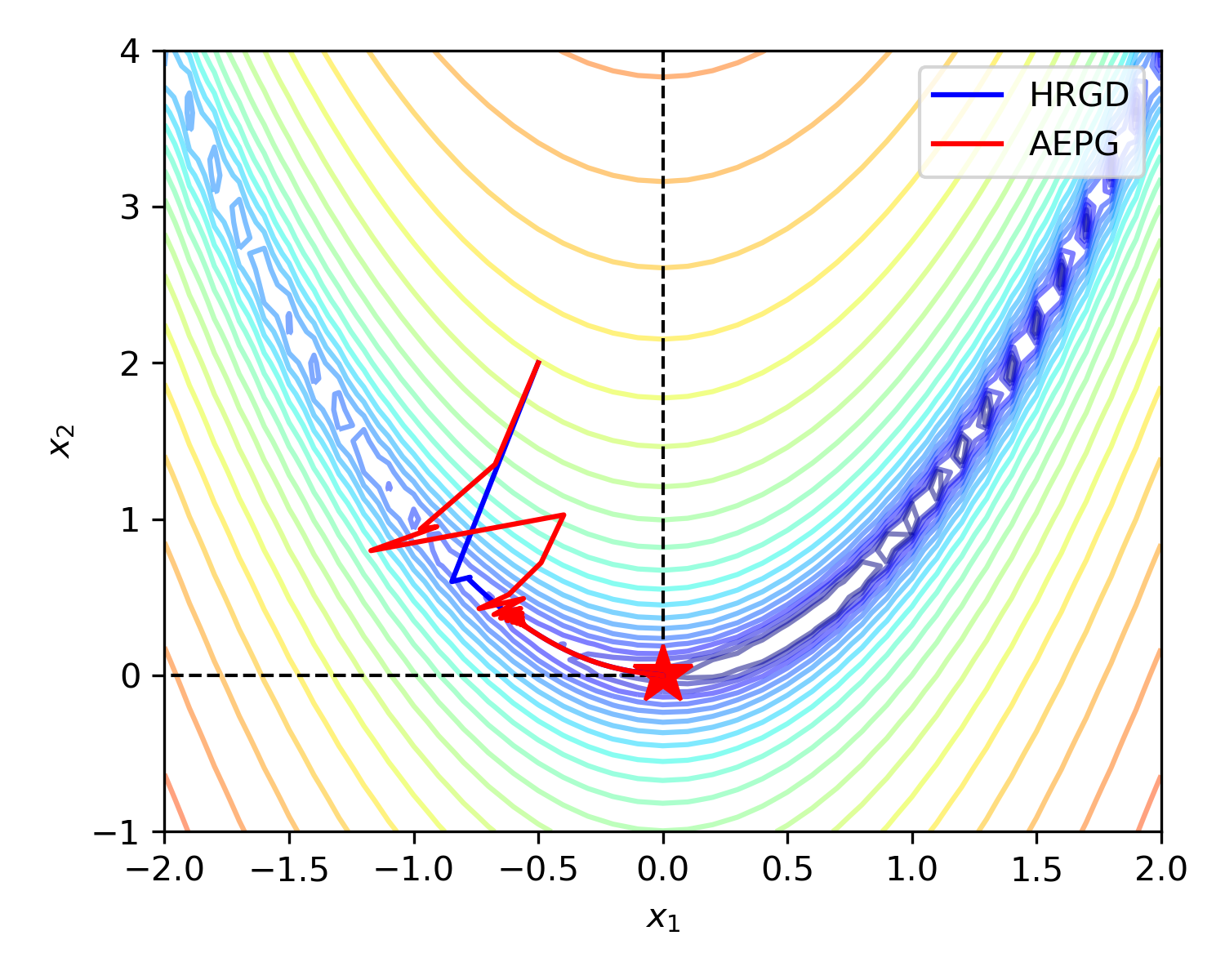}
\caption{Rosenbrock}
\end{subfigure}%
\caption{Contour plot and trajectories of AEPG and HRGD on two constrained optimization problems: the quadrative problem with $\alpha=10$ (a), and the Rosenbrock problem with $\alpha=100$ (b). In each plot, the red star represents the minimum point.}
\label{fig:hes}
\end{figure}

\subsubsection{Convex objectives}
We begin with the constrained quadratic problem:
\begin{align*}
\min\quad &f(x_1,x_2)=(x_1-1)^2+\alpha(x_2-1)^2,\\
\text{s.t.}\quad &(x_1+0.5)^2+(x_2-1)^2\leq1,
\end{align*}
where $\alpha$ is a positive constant. With the provided constraint, the minimum value of $f$ is $0.25$ achieved at $(0.5, 1)$. In our experiments, we vary the value of $\alpha$, which corresponds to the condition number of $\nabla^2 f$, and solve the problems using  HRGD (\ref{eq:pre-gd}) and AEPG (Algorithm \ref{alg2}) with $T_k=\nabla^2 h(x_k)^{-1}$. Here, we set 
\[
h(x_1, x_2)=u\ln u-u, \quad u:=u(x_1, x_2)=1-(x_1+0.5)^2-(x_2-1)^2.
\]
The initial point is set at $(-1, 1.8)$. The results are presented in Table \ref{tb:quad}, and the trajectories of the two methods are visualized in Figure \ref{fig:hes} (a).

\begin{table}[ht]
\caption{Number of iterations and computational time (in seconds) required by HRGD and AEPG to achieve $\epsilon$ accuracy on the constrained quadratic problem. Here $\alpha$ controls the condition number of $\nabla^2 f$; the stopping criteria is $|f(x)-f^*|<\epsilon$; Ratio = number of iterations (computational time) of HRGD / number of iterations (computational time) of AEPG.} 
\centering 
\begin{tabular}{l c rrr rrr} 
\toprule
\multirow{2}{*}{\bf $\alpha$} & 
\multirow{2}{*}{\bf $\epsilon$} &  
\multicolumn{3}{c}{\bf Number of iterations} &
\multicolumn{3}{c}{\bf Computational time (s)}\\ \cmidrule(lr){3-5} \cmidrule(lr){6-8}
&& HRGD &  AEPG  & Ratio & HRGD &  AEPG  & Ratio \\ \cmidrule(lr){1-8}
1      &  1e-7  &  416      & 103    & 4   & 0.0234  & 0.0075 & 3 \\
10     &  1e-6  &  3175     & 47     & 67  & 0.1592  & 0.0046 & 35 \\
100    &  1e-5  &  23120    & 723    & 32  & 1.1257  & 0.0533 & 21 \\
1000   &  1e-4  &  14190    & 1715   & 8   & 2.6121  & 0.1349 & 19 \\
10000  &  1e-3  &  147284   & 5075   & 29  & 14.9539 & 0.2805 & 53 \\
\bottomrule
\end{tabular}
\label{tb:quad} 
\end{table}


\subsubsection{Nonconvex objectives}
We then consider the benchmark 2D-Rosenbrock function of the form 
$$
f(x_1, x_2) = (x_1-1)^2+\alpha(x_2-x_1^2)^2, 
$$
where $\alpha$ is a positive constant. This serves as a standard test case for optimization algorithms. The global minimum $(1, 1)$ lies within a long, narrow, parabolic flat valley. Finding the valley is simple, but pinpointing the actual minimum of the function proves less trivial. In this example, we examine the constrained problem:
\begin{align*}
\min\quad  f(x_1, x_2) = (x_1-1)^2+\alpha(x_2-x_1^2)^2,\quad \text{s.t. }\quad  x_1<0, \; x_2>0.
\end{align*}
Under the given constraint, the minimum value of $f$ is $1$ at $(0,0)$. We vary the value of $\alpha$, which corresponds to the condition number of $\nabla^2 f$, and solve the problem using HRGD and AEPG (Algorithm \ref{alg2}) with $T_k=\nabla^2 h(x_k)^{-1}$. Here, we use
\[
h(x_1, x_2)=K(-x_1)+K(x_2),\quad K(s)=s\ln s-s,
\]
leading to $\nabla^2 h(x)^{-1}=\text{diag}(-x_1, x_2)$. The initial point is set at $(-0.5,2)$. The results are presented in Table \ref{tb:rosen}. Numerically, we observe that with a suitably larger step size, AEPG still converges to the minimum point, 
as shown in Figure \ref{fig:hes} (b).  

\begin{table}[ht]
\caption{Number of iterations and computational time (in seconds) required by HRGD and AEPG to achieve $\epsilon$ accuracy on the constrained Rosenbrock problem. Here $\alpha$ controls the condition number of $\nabla^2 f$; the stopping criteria is $|f(x)-f^*|<\epsilon$; Ratio = number of iterations (computational time) of HRGD / number of iterations (computational time) of AEPG.} 
\centering 
\begin{tabular}{l c rrr rrr} 
\toprule
\multirow{2}{*}{\bf $\alpha$} & 
\multirow{2}{*}{\bf $\epsilon$} &  
\multicolumn{3}{c}{\bf Number of iterations} &
\multicolumn{3}{c}{\bf Computational time (s)}\\ \cmidrule(lr){3-5} \cmidrule(lr){6-8}
&& HRGD &  AEPG  & Ratio & HRGD &  AEPG  & Ratio \\ \cmidrule(lr){1-8}
1      &  1e-7  &  7896     & 4802    & 2  & 0.2913 & 0.2257 & 1 \\
10     &  1e-6  &  7935     & 1956    & 4  & 0.2604 & 0.1294 & 2 \\
100    &  1e-5  &  8712     & 689     & 12  & 0.3606 & 0.0572 & 6 \\
1000   &  1e-4  &  28705    & 1327    & 21 & 0.8766 & 0.0843 & 10 \\
10000  &  1e-3  &  226524   & 2813    & 80 & 6.6848 & 0.1436 & 46 \\
\bottomrule
\end{tabular}
\label{tb:rosen} 
\end{table}

\subsubsection{$D$-Optimal Design Problem}  


Consider the problem of estimating a vector $x\in\R^m$ from measurements $y\in\R^n$ given by the relationship  
$$
y=Ux+\delta, \quad \delta \sim \mathcal{ N}(0, 1),
$$
where $U=[u_1,\cdots,u_n]^\top$ is a matrix that contains $n$ test (column) vectors $u_i\in\R^m$. During the experiment design phase, a reasonable goal is to minimize the covariance matrix, which is proportional to
$(U^\top U)^{-1}$. Using the D-optimality criteria, the problem is formulated as a minimum determinant problem \cite{VBW98}:
\begin{equation}\label{D}
\begin{aligned}
\min_{\theta}\quad & L(\theta):=\log \det \left(\sum_{i=1}^{n}\theta_iu_iu_i^{\top}\right)^{-1}, \\
\text{s.t.} \quad & \sum_{i=1}^{n}\theta_i=1\quad \text{and}\quad \theta_i\geq0,\quad i\in[n].
\end{aligned}
\end{equation}
This is a convex problem with the unit simplex as the feasible region. In computational geometry, the $D$-optimal design problem arises as a dual problem of the minimum volume covering ellipsoid (MVCE) problem and finds applications in  computational statistics and data mining \cite{VBW98}.

To apply AEPG (Algorithm \ref{alg2}) to solve \eqref{D}, we define the Hessian matrix by $\nabla^2 h(\theta_k)$, with 
\[
h(\theta)=\sum_{i=1}^n K(\theta_i),\quad K(s)=s\ln s-s.
\]
From this, we have $\nabla^2 h(\theta)^{-1}={\rm diag}(\theta_1,...,\theta_n)$. Note that $B=[1,1,...,1]$, resulting in the preconditioning matrix, as defined by \eqref{H}, taking the form: 
\begin{equation}\label{TD}
\begin{aligned}
T_k&=
\nabla^2 h(\theta_k)^{-1} - \nabla^2 h(\theta_k)^{-1}B^\top (B \nabla^2 h(\theta_k)^{-1}B^\top)^{-1}B 
\nabla^2 h(\theta_k)^{-1} \\
&={\rm diag}(\theta_{k,1},\cdots,\theta_{k,n}) -\theta_k\theta_k^\top, 
\end{aligned}
\end{equation}
where $B \nabla^2 h(\theta_k)^{-1}B^\top=\sum_{i=1}^{n}\theta_{k,i}=1$ is used.
Notably, from the structure of the preconditioning matrix in (\ref{TD}), it is evident that the computation of AEPG is independent of $m$ (the dimension of the test vectors $u_i$). Hence, AEPG is well-suited for solving D-optimal design problems constructed with high-dimensional test vectors. 

Several alternative algorithms have been proposed for solving \eqref{D}, such as  the interior point method and the Frank--Wolfe (FW) method \cite{AS+08}. While the interior point method requires the second-order derivative of $f$, the FW method is a first-order gradient method. To make the comparison more convincing, we also consider FW with away steps (FW-away), an effective strategy that enhances the vanilla FW algorithm's convergence speed and solution accuracy. Further details on applying FW and FW-away to solve the D-optimal design problem can be found in  \cite[Chapter 5.2.7]{BC+22}. 

In our experiments, we maintain $n=1000$ (number of test vectors) and compare AEPG (Algorithm \ref{alg2}) with other algorithms for various values of $m$:  $10, 30, 50, 80, 100, 200, 300, 400, 500$. We generate test vectors $u_i$ using independent random Gaussian distributions with zero mean and unit variance. The initial point is set as $\theta_0=(\frac{1}{n},...,\frac{1}{n})$. All computations are performed using  Python 3.7 on a 2 GHz PC with 16 GB Memory.

\textbf{Comparison with the interior point method (IPM).}
For cases where $m=10, 30, 50, 80, 100$, we conduct a comparison analysis between  AEPG with IPM \footnote{The interior point method is applied through the Python package PICOS \cite{SS22}. 
}. Specifically, we establish the stopping criteria as $|L(\theta)-L^*|<10^{-7}$ and compare the computation time required for both methods to meet the stopping criteria. The summarized results in Table~\ref{tb2} reveal that the computation time of AEPG is significantly less than that of IPM, especially for relatively large $m$ ($m>10$). 

\begin{table}[ht]
\caption{Comparison of computational time (in seconds) between AEPG and IPM for the D-optimal design problem. The datasets have varying dimensions of test vectors $m$ with a fixed number of test vectors $n=1000$. Ratio = computational time of IPM / computational time of AEPG.} 
\centering 
\begin{tabular}{c c r c r} 
\toprule
\multirow{2}{*}{\bf $m$} & 
\multirow{2}{*}{\bf $n$} &  
\multicolumn{3}{c}{\bf Computational time (s)}\\ \cmidrule(lr){3-5} 
&& IPM   & AEPG  & Ratio\\ \cmidrule(lr){1-5}
10   & 1000   & 1.39    &  2.84   & 0.5 \\
30   & 1000   & 6.15    &  2.93   & 2 \\
50   & 1000   & 21.33   &  3.02   & 7 \\
80   & 1000   & 112.92  &  2.63   & 43 \\
100  & 1000   & 253.37  &  2.44   & 104 \\
\bottomrule
\end{tabular}
\label{tb2} 
\end{table}

\textbf{Comparison with the Frank-Wolfe (FW) method.}
For scenarios where $m=200, 300, 400, 500$, 
the IPM failed to solve the problems.
Consequently, we compare AEPG with the FW method and FW-away method. The results are visually represented in Figure \ref{fig:Doptimal}. Across all cases, the vanilla FW algorithm fails to converge  to solutions meeting the stopping criteria, and AEPG consistently requires  less time than FW-away to reach the minimum value, particularly when $m=300, 400, 500$.

\begin{figure}[ht]
\begin{subfigure}[b]{0.25\linewidth}
\centering
\includegraphics[width=1\linewidth]{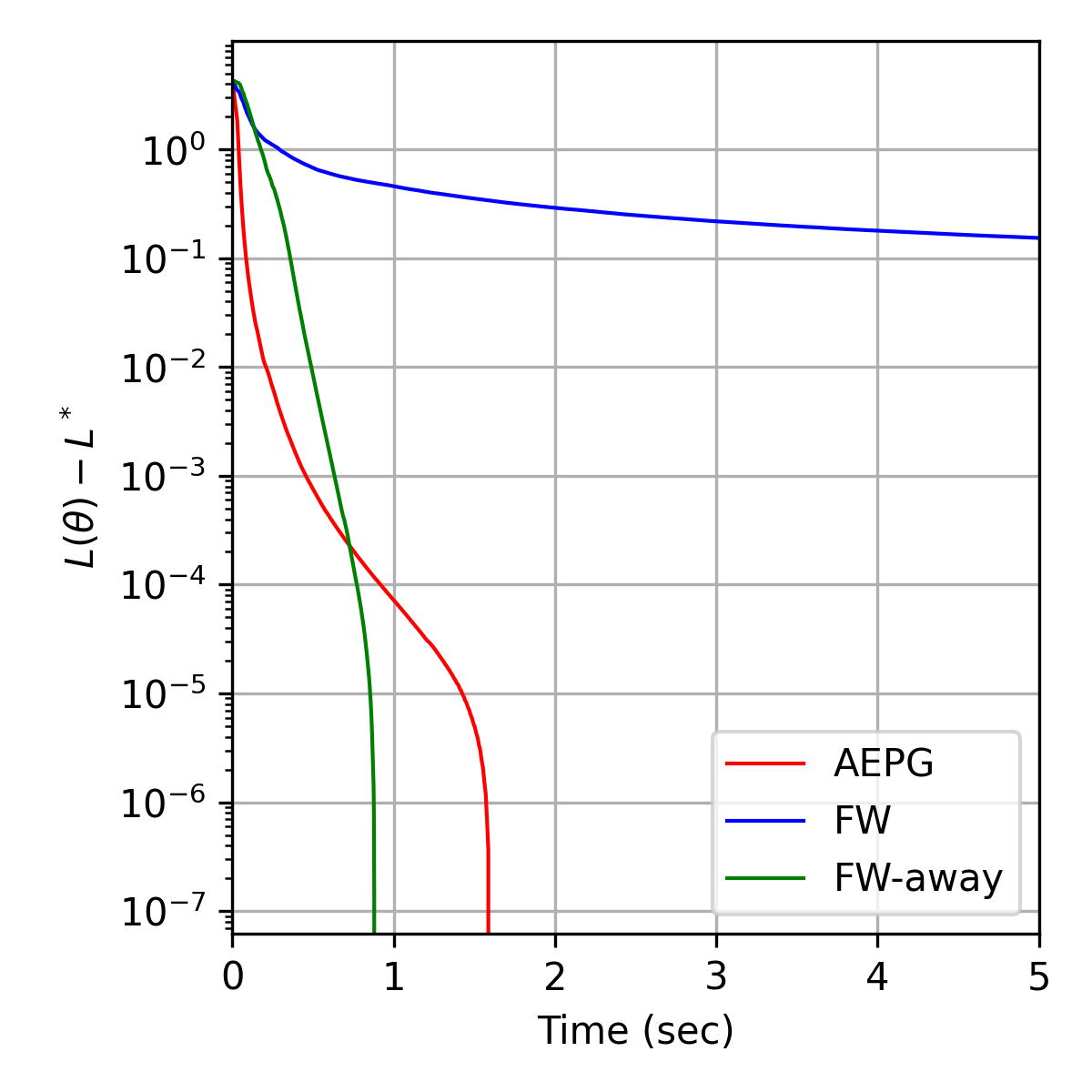}
\caption{$m=200$}
\end{subfigure}%
\begin{subfigure}[b]{0.25\linewidth}
\centering
\includegraphics[width=1\linewidth]{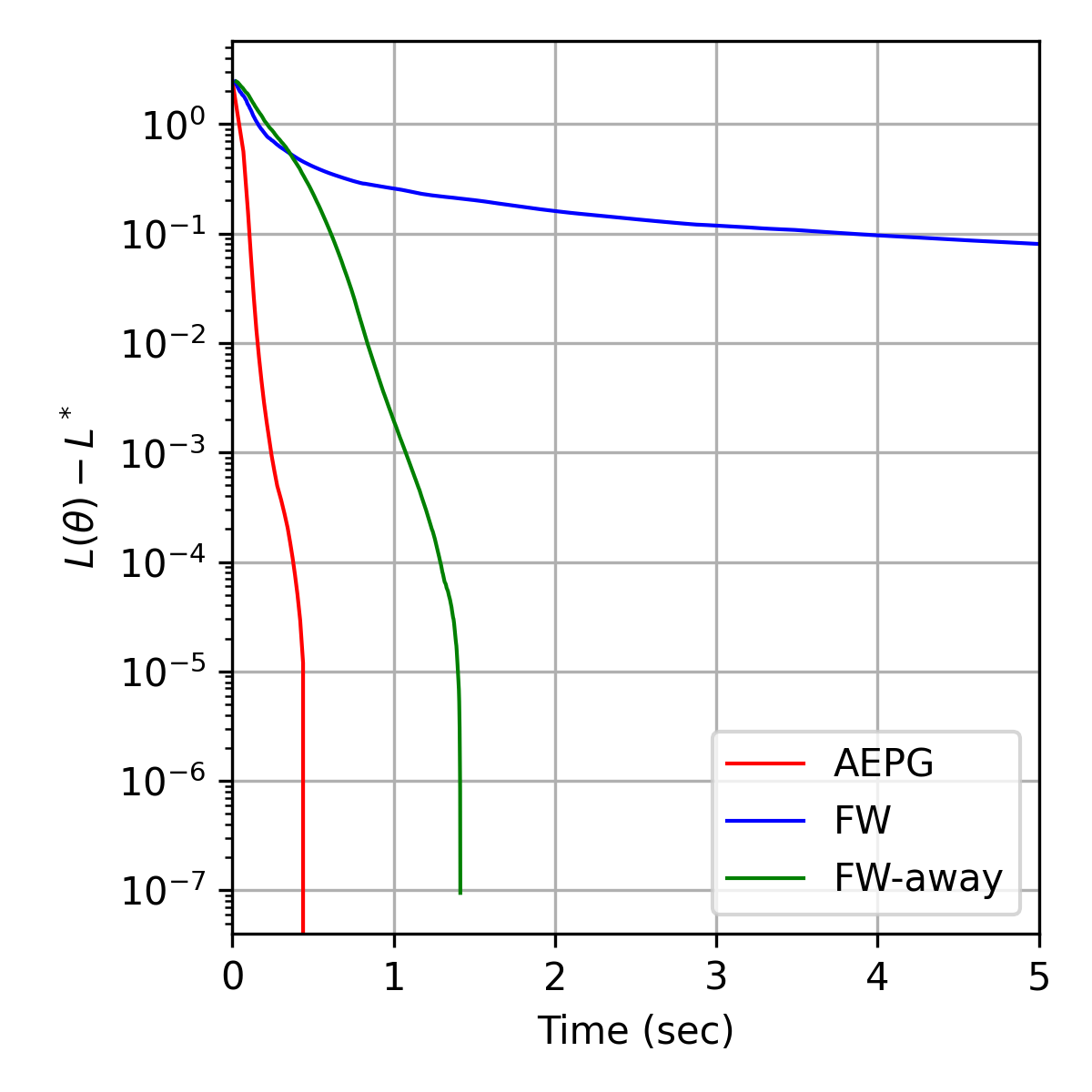}
\caption{$m=300$}
\end{subfigure}%
\begin{subfigure}[b]{0.25\linewidth}
\centering
\includegraphics[width=1\linewidth]{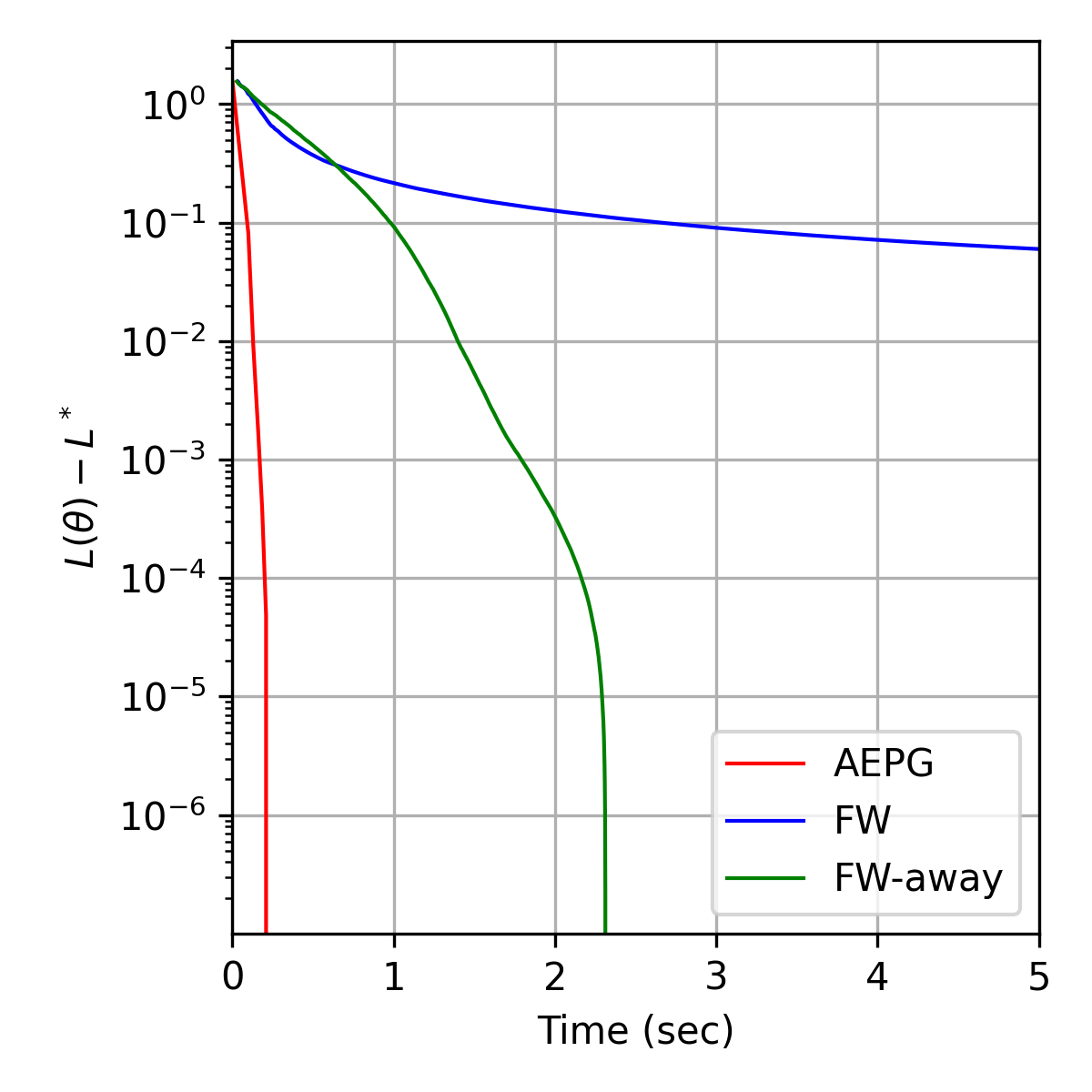}
\caption{$m=400$}
\end{subfigure}%
\begin{subfigure}[b]{0.25\linewidth}
\centering
\includegraphics[width=1\linewidth]{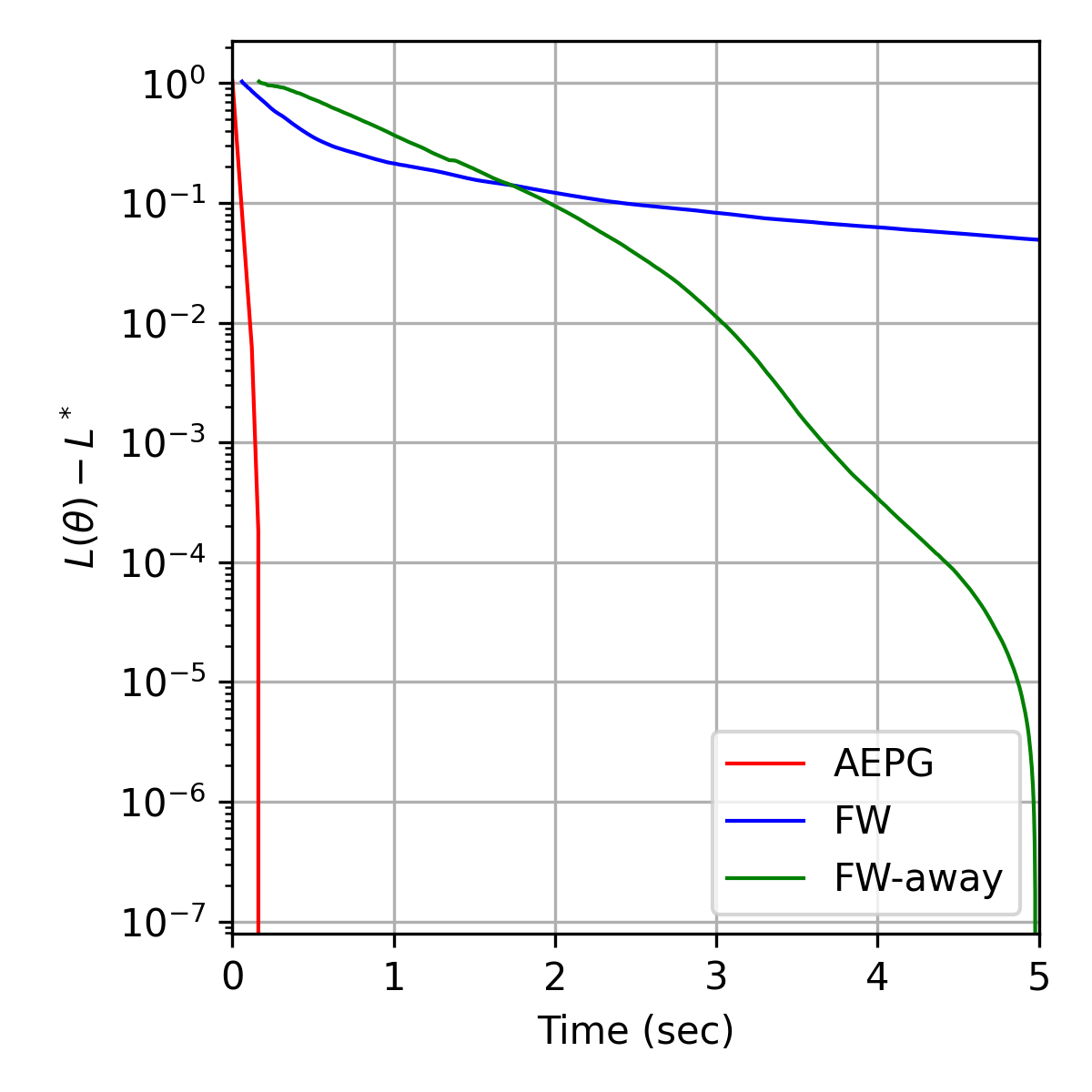}
\caption{$m=500$}
\end{subfigure}%
\caption{Comparison of computational time (in seconds) between AEPG and the FW/FW-away method for the D-optimal design problem. The datasets have varying dimensions of test vectors $m$ and a fixed number of test vectors $n=1000$}
\label{fig:Doptimal}
\end{figure}

In both comparative experiments, we observe that unlike the FW-away algorithm and the IPM, whose computation time increases as $m$ gets larger, the computation time of AEPG remains nearly constant across all cases. The advantage of AEPG over the IPM and the FW-away algorithm becomes increasingly evident as $m$ gets larger.

\subsection{Wasserstein natural gradient}\label{ex-wng}
Consider a 2-dimensional Gaussian mixture model for $\rho$ defined as: 
$$
\rho(x;\theta)=w\mathcal{N}(x;(\theta_1,3),I)+(1-w)\mathcal{N}(x;(\theta_2,2),I),
$$
where $w\geq0$ is a weight factor. In accordance with \cite{NLY22}, we formulate the data fitting problem: 
$$
\min_\theta\bigg\{L(\theta)=\frac{1}{2}\int_{\Omega}|\rho(x;\theta)-\rho^*(x)|^2dx\bigg\},
$$
where $\Omega$ is a compact domain, and $\rho^*(x)$ is the observed reference density function. In our experiments, we set $w=0.05$, $\Omega=[0,5]^2$, and $\rho^*(x)=\rho(x;(1,3))$. The initial point is set at $(4,4.2)$. 

We apply the least-squares formulation (\ref{rhov}) and (\ref{rhof}) to compute the Wasserstein natural gradient and compare the performance of WNGD (\ref{eq:pre-gd}) and AEPG (Algorithm \ref{alg2}) with $T_k=G(\theta_k)^{-1}$, where $G(\theta)$ is the Wasserstein information matrix~\eqref{eq:G}. The trajectories are presented in Figure \ref{fig:rho}(b). Notably, both WNGD and AEPG successfully  locate the global minima, with AEPG converging in fewer iterations. Additionally, we showcase the trajectories of GD and AEGD (using standard gradient) in Figure \ref{fig:rho}(a), observing that both methods with standard gradient get stuck at a local minimum.

\begin{figure}[ht]
\begin{subfigure}[b]{0.5\linewidth}
\centering
\includegraphics[width=1\linewidth]{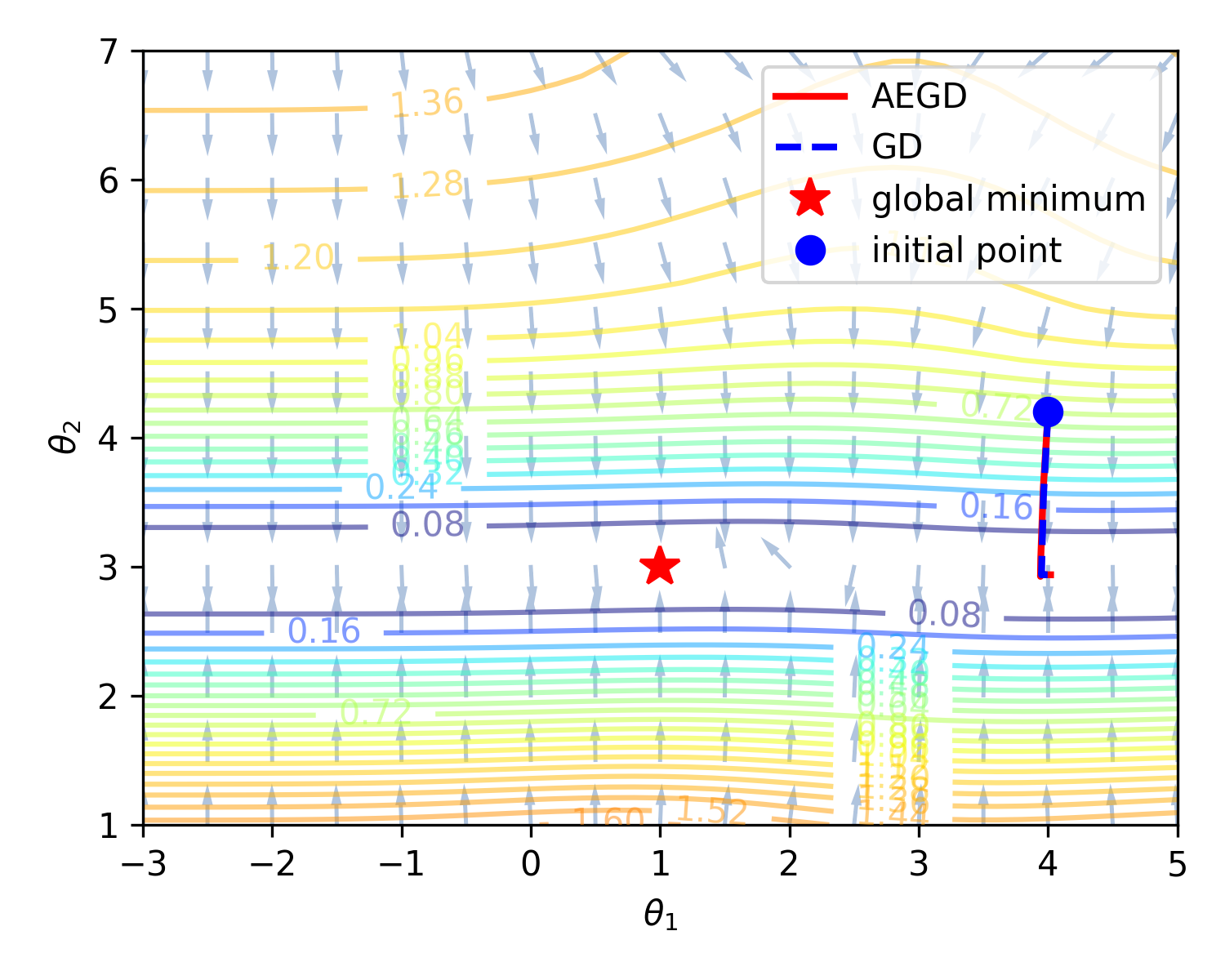}
\caption{Standard gradient}
\end{subfigure}%
\begin{subfigure}[b]{0.5\linewidth}
\centering
\includegraphics[width=1\linewidth]{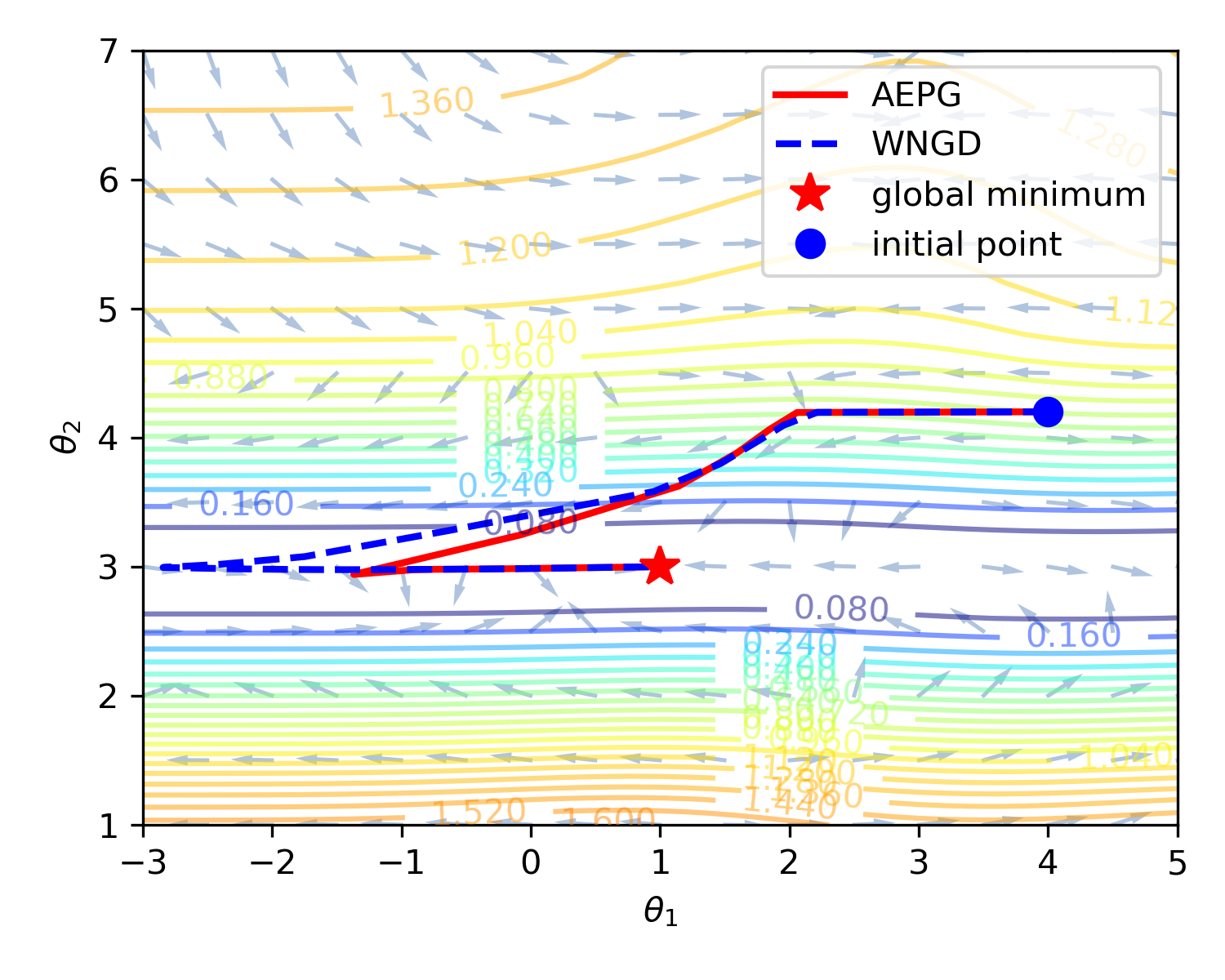}
\caption{Wasserstein natural gradient}
\end{subfigure}%
\caption{Gaussian mixture model: level sets, vector fields, and convergent paths using (a) methods with standard gradient and (b) methods with the Wasserstein natural gradient.}
\label{fig:rho}
\end{figure}
 %
\section{Discussion}\label{discuss}
This research introduces a unified framework for the application of AEGD (Adaptive Gradient Descent with Energy) techniques,  using a general preconditioning descent direction to address a class of constrained optimization problems. The key idea is to incorporate the advantages of adapting the descent direction to the problem's geometry  and adjusting the step size through an energy variable. 

Theoretical insights reveal that AEPG (Adaptive Energy Preconditioned Gradient) is unconditionally energy stable, independent of the step size. Under the condition of a suitably small base step size,  we establish that AEPG is guaranteed to find the minimum value of the objective function.   Convergence rates are derived for three types of objective functions: general differentiable functions, nonconvex functions satisfying Polyak--Lojasiewic's (PL) condition, and convex functions when 
the descent direction is preconditioned by a positive definite matrix against a possible projection matrix. 

We investigate two application scenarios  where optimization problems with explicit or implicit constraints are considered. In one instance,  optimization problems with parameters in a convex set are addressed by endowing the feasible set with a Riemannian metric, following  
the strategy outlined in \cite{ABB04}. Extension to cases with linear equality constraints is achieved by adjusting the preconditioning matrix through a projection operator, with convergence theories established under mild assumptions. Another application pertains to optimization problems over the probability density space with the Wasserstein metric. To efficiently compute the natural gradient $G(\theta)^{-1}\nabla l(\theta)$ in AEPG,   we adopt the strategy proposed in \cite{NLY22},  treating the natural gradient as the solution to a least-squares problem. 

Numerical results show that AEPG outperforms vanilla preconditioned gradient descent algorithms such as HRGD and NGD, particularly on ill-conditioned or nonconvex problems. Notably, these results indicate that the choice of the preconditioning matrix not only impacts the convergence rate but influences the stationary point where the iterates converge within a nonconvex optimization landscape.

Optimal choices for Riemannian metrics remain unclear for manifolds with intricate structures. Good options should probably combine specific problem characterisitcs and applications. Identifying a suitable preconditioning matrix that accelerates convergence without introducing computational overhead is crucial. Moreover, the development of efficient methods for computing the natural gradient in large-scale optimization problems is imperative.


\section*{Declarations} 

\subsection*{Funding}
This work has been partially supported by National Science Foundation (NSF) grants DMS1812666 and DMS2135470 to Hailiang Liu and Xuping Tian; Air Force Office of Scientific Research (AFOSR), Multi-University Research Initiative (MURI) FA 9550 18-1-0502 grant to Levon Nurbekyan; Office of Naval Research (ONR) grant N00014-24-1-2088, and National Science Foundation (NSF) grant DMS-1913129 to Yunan Yang.
 
\subsection*{Data Availability}
The synthetic data supporting Table 3 and Figure 2 are included at \url{https://github.com/txping/AEPG}.

\subsection*{Conflict of interest}
The authors have no relevant financial or non-financial interests to disclose.

 %
\appendix
\renewcommand{\theequation}{\thesection.\arabic{equation}}

\section{Technical proofs}\label{pf}
In this section, we provide proofs for some theoretical results mentioned in  Section \ref{AENG}:
\subsection{Proof of Theorem \ref{cvg}}\label{pf-cvg}
To analyse the convergence behavior of AEPG (\ref{aeng0}), we reformulate it as 
\begin{equation}\label{eq:energy}
 \theta_{k+1}= \theta_k- \eta_k A_k^{-1}\nabla L(\theta_k), \quad \eta_k:=\eta \frac{r_{k+1}}{l(\theta_k)}.
\end{equation} 
Using the $\alpha$-smoothness of $L$, we have 
\begin{equation}\label{lsmooth}
\begin{aligned}
L (\theta_{k+1}) & = L(\theta_k- \eta_k A_k^{-1}\nabla L(\theta_k))\\
& \leq L(\theta_k)-\eta_k \nabla L(\theta_k)^\top A_k^{-1}\nabla L(\theta_k)+\frac{\eta^2_k \alpha}{2}\|A_k^{-1}\nabla L(\theta_k)\|^2 \\
& \leq L(\theta_k)-\eta_k\|\nabla L(\theta_k)\|^2_{A^{-1}_k} + \frac{\eta^2_k \alpha }{2\lambda_1}\|\nabla L(\theta_k)\|^2_{A^{-1}_k}\\
& = L(\theta_k)- \eta_k\left(1 -\frac{\eta_k \alpha}{2\lambda_1} \right)\|\nabla L(\theta_k)\|^2_{A^{-1}_k}. 
\end{aligned}
\end{equation}
This ensures that $L(\theta_k)$ is strictly decreasing as long as $\|\nabla L(\theta_k)\|^2_{A^{-1}_k}\neq 0$, providing that 
\begin{equation}\label{eta_bounds}
0<\eta_k< \frac{2\lambda_1}{\alpha}.    
\end{equation}
Hence we have $L(\theta_k)$ converges as $k\to \infty$. 
We further sum (\ref{lsmooth}) over $k$  to get
$$
\sum_{k=0}^{\infty}\eta_k\left(1 -\frac{\eta_k \alpha}{2\lambda_1} \right)\|\nabla L(\theta_k)\|^2_{A^{-1}_k}\leq \sum_{k=0}^{\infty} \left(L(\theta_k)-L (\theta_{k+1}) \right) \leq L(\theta_0)-\inf L(\theta_k)<\infty.
$$
This implies
\begin{equation}\label{eta-grad}
\lim_{k\to\infty} \eta_k\left(1 -\frac{\eta_k \alpha}{2\lambda_1} \right)\|\nabla L(\theta_k)\|^2_{A^{-1}_k} = 0. 
\end{equation}
With $\eta\leq\eta_s$, Lemma \ref{tau} guarantees that $\eta_k:= \eta \frac{r_{k+1}}{l(\theta_k)}\geq \frac{\eta r^*}{l(\theta_0)}$ when $r_0\geq\frac{l(\theta_0)-l^*}{\lambda_1}$. With $\eta\leq \eta_0=\frac{\lambda_1l^*}{\alpha r_0}$, we have $\eta_k := \eta \frac{r_{k+1}}{l(\theta_k)}\leq \frac{\lambda_1}{\alpha}$. 
These two bounds imply that $\eta_k\left(1 -\frac{\eta_k \alpha}{2\lambda_1} \right)\geq \frac{\eta r^*}{2l(\theta_0)}>0$, which together with (\ref{eta-grad}) ensures that $\|\nabla L(\theta_k)\|^2_{A^{-1}_k}\to 0$, hence $\|\nabla L(\theta_k)\|\to 0$ since $A^{-1}$ is positive definite. Therefore, $L(\theta_k)$ is guaranteed to converge to a local minimum value of $L$.

\subsection{Proof of Theorem \ref{thm2}}\label{pf2}
(ii)  Using the $\alpha$-smoothness assumption and similar derivation as \eqref{lsmooth}, we have
\begin{align}
L(\theta_{k+1}) - L(\theta_k)\notag
&\leq - \eta_k\left(1 -\frac{\eta_k \alpha}{2\lambda_1} \right)\|\nabla L(\theta_k)\|^2_{A^{-1}_k}\\\notag
&\leq -\frac{\eta_k}{2}\|\nabla L(\theta_k)\|^2_{A^{-1}_k}\quad\text{(since $\eta_k\leq\lambda_1/\alpha $)}\\\label{df}
& \leq -\frac{\eta_k}{2\lambda_n}\|\nabla L(\theta_k)\|^2. 
\end{align}    
Denote $w_k=L(\theta_k)-L(\theta^*)$, then the PL property of function $L$ implies
\begin{equation}\label{plw}
\frac{1}{2}\|\nabla L(\theta_k)\|^2\geq\mu(L(\theta_k)-L(\theta^*))=\mu w_k.    
\end{equation}
With this property, \eqref{df} can be written as
\begin{equation}\label{dwk}
w_{k+1}-w_k\leq -\frac{\mu\eta_k}{\lambda_n}w_k\quad\Rightarrow\quad w_{k+1}\leq \Big(1-\frac{\mu\eta_k}{\lambda_n}\Big) w_k.   
\end{equation}
By induction, 
\begin{align*}
w_k &\leq 
w_0\prod_{j=0}^{k-1}\Big(1-\frac{\mu\eta_j}{\lambda_n}\Big)
= w_0\exp\left(\sum_
{j=0}^{k-1} {\rm log}\Big(1-\frac{\mu\eta_j}{\lambda_n}\Big)\right) 
\leq w_0\exp\left(-\mu \sum_
{j=0}^{k-1} \frac{\eta_j}{\lambda_n}\right).
\end{align*}
Note for $j<k$ that 
$\eta_j=\eta \frac{r_{j+1}}{l(\theta_j)}\geq \eta \frac{r_{k}}{l(\theta_0)},
$
hence, 
$$
w_k \leq w_0\exp\Big(-\frac{c_0k r_k}{\lambda_n}\Big), \quad c_0=\frac{\mu \eta }{l(\theta_0)}. 
$$
To ensure convergence of $\theta_k$, we use $\alpha$-smoothness for $L$ and scheme \eqref{eq:energy} to get
\begin{align*}
L(\theta_{k+1})-L(\theta_k)
&\leq \nabla L(\theta_k)^\top (\theta_{k+1}-\theta_k)+\frac{\alpha}{2}\|\theta_{k+1}-\theta_k\|^2\\
&\leq -\frac{1}{\eta_k}\|\theta_{k+1}-\theta_k\|^2_{A_k}+\frac{\alpha}{2}\|\theta_{k+1}-\theta_k\|^2\\
&\leq (-\frac{\lambda_1}{\eta_k}+\frac{\alpha}{2})\|\theta_{k+1}-\theta_k\|^2\\
&\leq -\frac{\lambda_1}{2\eta_k}\|\theta_{k+1}-\theta_k\|^2\quad\text{(since $\eta_k\leq\lambda_1/\alpha$)},
\end{align*}
which further implies
\begin{equation}\label{wdiff}
w_k-w_{k+1}\geq \frac{\lambda_1}{2\eta_k}\|\theta_{k+1}-\theta_k\|^2.  
\end{equation}
The PL property (\ref{plw}) when combined with (\ref{eq:energy}) gives  
\begin{equation}\label{1/w}
\|A_k(\theta_{k+1}-\theta_k)\|^2
= \eta^2_k\|\nabla L(\theta_k)\|^2
\geq 2\mu\eta^2_kw_k \quad\Rightarrow\quad \frac{1}{\sqrt{w_k}}\geq\frac{\sqrt{2\mu}\eta_k}{\lambda_n\|\theta_{k+1}-\theta_k\|}.    
\end{equation}
Using $w_k >w_{k+1}$, we have
\begin{align*}
\sqrt{w_k}-\sqrt{w_{k+1}}
&\geq \frac{1}{2\sqrt{w_k}}(w_k-w_{k+1})
\geq \frac{\sqrt{2\mu}\lambda_1}{4\lambda_n}\|\theta_{k+1}-\theta_k\|.
\end{align*}
Here the last inequality is by (\ref{wdiff}) and (\ref{1/w}).
Taking summation over $k$ gives 
$$
\sum_{k=0}^\infty\|\theta_{k+1}-\theta_k\|\leq \frac{4\lambda_n}{\sqrt{2\mu}\lambda_1} \sqrt{w_0}.
$$
This yields (\ref{ctheta}), which ensures the convergence of $\{\theta_k\}$. 

(iii) With the convexity assumption on $L$, we have 
\begin{align*}
w_k:=L(\theta_k) - L(\theta^*) 
& \leq \nabla L(\theta_k)^\top (\theta_k-\theta^*) \leq \|\nabla L(\theta_k)\|\|\theta_k-\theta^*\|,
\end{align*}
where we used the Cauchy-Schwarz inequality. We claim that for convex $f$, 
\begin{align}\label{cc}
 \|\theta_k-\theta^*\| \leq \|\theta_0-\theta^*\|.  
\end{align}
The proof of this claim is deferred  to the end of this subsection. Thus we have 
\begin{equation}\label{cond2}
\|\nabla L(\theta_k)\|
\geq \frac{w_k}{\|\theta_k-\theta^*\|}
\geq\frac{w_k}{\|\theta_0-\theta^*\|}.
\end{equation}
This when combined with  \eqref{df} (since $\eta_k\leq \lambda_1/\alpha$) leads to 
\begin{equation}\label{wkk}
w_{k+1} \leq w_k -\frac{\eta_k}{2\lambda_n\|\theta_0-\theta^*\|^2}w_k^2.   
\end{equation}
This implies $w_k\geq w_{k+1}$.  Multiplying $\frac{1}{w_{k+1}w_k}$ on both sides gives
$$
\frac{1}{w_k} \leq \frac{1}{w_{k+1}}-\frac{\eta_k}{2\lambda_n\|\theta_0-\theta^*\|^2}\frac{w_k}{w_{k+1}}
\leq \frac{1}{w_{k+1}}-\frac{\eta_k}{2\lambda_n\|\theta_0-\theta^*\|^2}.
$$
Using $\eta_j=\eta \frac{r_{j+1}}{l(\theta_j)}\geq \eta \frac{r_{k}}{l(\theta_0)}$, we proceed to obtain 
\begin{align*}
\frac{1}{w_k} 
& \geq \frac{1}{w_{k-1}} + \frac{\eta_{k-1}}{2\lambda_n\|\theta_0-\theta^*\|^2}\\
& \geq \frac{1}{w_0} + \frac{1}{2\lambda_n\|\theta_0-\theta^*\|^2}\sum_{j=0}^{k-1}\eta_j\\
& \geq \frac{k\eta r_k }{2\lambda_n l(\theta_0)\|\theta_0-\theta^*\|^2}.
\end{align*}
Hence for $\max_{j<k} \eta_j \leq \lambda_1/\alpha$, we have
$$
L(\theta_k)-L(\theta^*)=w_k \leq \frac{c_1\lambda_n\|\theta_0-\theta^*\|^2}{k r_k},\quad c_1=\frac{2l(\theta_0)}{\eta}.
$$
Finally, we prove claim (\ref{cc}). We proceed with
\begin{align*}
\theta_{k+1} -\theta^* & = \theta_{k}-\theta^* - \eta_k A_k^{-1} \nabla L(\theta_k) \\
& =\theta_{k} -\theta^* - \eta_k A_k^{-1}(\nabla L(\theta_k) -\nabla L(\theta^*))\\
&=\theta_{k}-\theta^* -\eta_k A_k^{-1} \left( \int_0^1 \nabla^2 L(\theta^*+s(\theta_k-\theta^*))ds=:B_k\right)
(\theta_k-\theta^*).
\end{align*}
Denote $d_k:=\|\theta_k-\theta^*\|$, we have
\begin{align*}
d_{k+1}& =\|(I-\eta_k A_k^{-1} B_k)(\theta_k-\theta^*)\|\\
& \leq \max_{0\leq \xi \leq \frac{\alpha}{\lambda_1}} |1-\eta_k \xi|d_{k}.
\end{align*}
For $\eta_k \leq \lambda_1 /\alpha$, we have
\begin{align*}
d_{k+1}& \leq  d_{k}, 
\end{align*}
hence $d_k\leq d_0$ for any integer $k$. This completes the proof.

\subsection{Proof of Theorem \ref{thmT}}\label{pf-thmT}

(3) The proof is similar to the proof for Theorem \ref{cvg}. Using $\alpha$-smoothness, we have 
\begin{equation}\label{lsmooth+}
\begin{aligned}
L(\theta_{k+1}) & = L(\theta_k- \eta_k T_k\nabla L(\theta_k))\\
& \leq L(\theta_k)-\eta_k \nabla L(\theta_k)^\top T_k\nabla L(\theta_k)+\frac{\eta^2_k\alpha}{2}\|T_k\nabla L(\theta_k)\|^2.
\end{aligned}
\end{equation}
Since $T_k=P_kG_k^{-1}=G_k^{-1}P_k^\top$, we have 
\begin{align}\label{LPk}\notag
L(\theta_{k+1}) & \leq 
L(\theta_k) -\eta_k \|P_k^\top\nabla L(\theta_k)\|^2_{G_k^{-1}}
+ \frac{\alpha \eta_k^2 }{2\lambda_1}\|P_k^\top\nabla L(\theta_k)\|^2_{G_k^{-1}}\\
& = L(\theta_k)- \eta_k\left(1 -\frac{\eta_k \alpha}{2\lambda_1} \right)\|P_k^\top\nabla L(\theta_k)\|^2_{G_k^{-1}}.
\end{align}
This ensures that $L(\theta_k)$ is strictly decreasing as long as $\|P_k^\top\nabla L(\theta_k)\|^2_{G^{-1}_k}\neq 0$, providing that 
(\ref{eta_bounds}) holds. For the rest of the proof, we refer the readers to (\ref{pf-cvg}).

(4) We now turn to estimate convergence rates, while we use  
\begin{equation*}
\quad P_k= I- G_k^{-1} B^\top (B G_k^{-1} B^\top)^{-1}B
\end{equation*}
so that $B P_k=0$. 

(i) In the proof of (i) for Theorem \ref{thm2}, we  replace  $v_k=G^{-1}_k\nabla l(\theta_k)$ by $G_k^{-1} P_k^\top \nabla l(\theta_k)$, thus obtain the stated convergence bound for $\min_{j<k}\|P_j^\top \nabla L(\theta_j)\|^2$.

(ii) We first show convergence of $\theta_k$. 
Using $\eta_k \leq \frac{\lambda_1}{\alpha}$, we further bound (\ref{LPk}) by
\begin{equation}\label{Lsmo}
L(\theta_{k+1})\leq L(\theta_k)- \frac{\eta_k}{2} \|P_k^\top\nabla L(\theta_k)\|^2_{G_k^{-1}}.
\end{equation} 
This and
$$
\theta_{k+1}-\theta_k=-\eta_kG_k^{-1}P_k^\top\nabla L(\theta_k) \in K,\quad K: =\{p|\quad B p=0\}
$$ 
lead to 
\begin{equation}\label{dft1}
L(\theta_{k+1})-L(\theta_k)\leq -\frac{\lambda_1}{2\eta_k}\|\theta_{k+1}-\theta_k\|^2\quad\Rightarrow\quad w_k-w_{k+1}\geq \frac{\lambda_1}{2\eta_k}\|\theta_{k+1}-\theta_k\|^2. 
\end{equation}
The projected PL property (\ref{PPL}) when combined with (\ref{dft1}) gives  
\begin{equation}\label{dft2}
\|G_k(\theta_{k+1}-\theta_k)\|^2
= \eta^2_k\|P_k^\top\nabla L(\theta_k)\|^2
\geq 2\mu\eta^2_kw_k \quad\Rightarrow\quad \frac{1}{\sqrt{w_k}}\geq\frac{\sqrt{2\mu}\eta_k}{\lambda_n\|\theta_{k+1}-\theta_k\|}.    
\end{equation}
Combining (\ref{dft1}) with  (\ref{dft2}) gives 
$$
\sqrt{w_k}-\sqrt{w_{k+1}}
\geq \frac{1}{2\sqrt{w_k}}(w_k-w_{k+1}) 
\geq \frac{\sqrt{2\mu}\lambda_1}{4\lambda_n}\|\theta_{k+1}-\theta_k\|.
$$
This ensures that $\theta_k$ is a Cauchy sequence, hence 
$\theta_k\to \theta^*$ as $k\to \infty$.

Next we show the convergence rate of  $w_k=L(\theta_k)-L(\theta^*)$. From (\ref{Lsmo}) we have 
\begin{equation}\label{Lsmoo}
w_{k+1}-w_k = L(\theta_{k+1})- L(\theta_k)
\leq -\frac{\eta_k}{2\lambda_n}\|P_k^\top\nabla L(\theta_k)\|^2
\end{equation}
for $\eta_k \leq \frac{\lambda_1}{\alpha}$. 
By the the projected PL condition (\ref{PPL}), we have
$$
\|P_k^\top\nabla L(\theta_k)\|^2 \geq 2\mu w_k, 
$$
hence 
$$
w_{k+1}\leq w_k\Big(1-\frac{\mu \eta_k}{\lambda_n}\Big),
$$
which is exactly  (\ref{dwk}), hence an  induction argument will imply
$$
w_k \leq w_0\exp\Big(-\frac{c_0k r_k}{\lambda_n}\Big), \quad c_0=\frac{\mu \eta }{l(\theta_0)}. 
$$
(iii) For linear constraint $q(\theta)=B\theta -b $, we have 
$\theta_{k+1}-\theta_k \in K=\{p|\quad B p=0\}$ with $B=\nabla q(\theta_k)$, hence $\theta_{k}-\theta^* \in K=\{p|\quad Bp=0\}$.
By the convexity of $L$, 
\begin{align*}
w_k:=L(\theta_k) - L(\theta^*) 
&\leq \nabla L(\theta_k)^\top (\theta_{k}-\theta^*) \\
&= \nabla L(\theta_k)^\top P_k(\theta_{k}-\theta^*) \\
&= (P_k^\top\nabla L(\theta_k))^\top (\theta_{k}-\theta^*) \\
&\leq \|P_k^\top\nabla L(\theta_k)\|\|\theta_{k}-\theta^*\|, 
\end{align*}
which implies
\begin{equation*}
\|P_k^\top\nabla L(\theta_k)\|
\geq \frac{w_k}{\|\theta_k-\theta^*\|}
\geq\frac{w_k}{\|\theta_0-\theta^*\|}.
\end{equation*}
This when combined with (\ref{Lsmoo}) gives
\begin{equation*}
w_{k+1}-w_k\leq -\frac{\eta_k}{2\lambda_n}\|P_k^\top\nabla L(\theta_k)\|^2 \leq 
-\frac{\eta_k w_k^2}{2\lambda_n \|\theta_0-\theta^*\|^2}.
\end{equation*}
This is the same as (\ref{wkk}), and the rest argument in the proof for Theorem \ref{thm2} (iii) applies here.

\section{Construction of $h$}\label{cons}
We present a construction of $h$ for $\M$ when ~\eqref{eq:M=suplevelU} holds with a smooth concave $U$. We define 
\[
h(\theta)=K(U(\theta)), \quad \theta \in \operatorname{int}(\M), 
\]
where $K: (0, \infty) \to \mathbb{R}$ is a smooth function. It is important to note that:
\begin{align}\label{hess}
& \nabla h(\theta) =K'(U(\theta))\nabla U(\theta), \\\notag
& \nabla^2 h(\theta) =K''(U(\theta))\nabla U(\theta)\otimes \nabla U(\theta)+K'(U(\theta))\nabla^2 U(\theta),
\end{align}
and it is clear that $K$ should be chosen to satisfy the following conditions: 
\begin{enumerate}
\item[(i)] $\lim_{s\to0^+}K'(s)=-\infty$,
\item[(ii)] $\forall s>0$, $K''(s)>0$,
\item[(iii)] $\forall s>0$, $K'(s) < 0$ when $U$ is not an affine function.
\end{enumerate}
Two common used functions for $K$ are $-\ln(s)$ and $ s\ln(s)-s$. Additional admissible choices can be found in ~\cite{ABB04}.  
\begin{lemma}
If $K$ is selected to satisfy conditions (i)-(iii) as mentioned above, the strict convexity of $h=K(U)$ can only be ensured if 
\begin{align}\label{con}
\xi\cdot \nabla U(\theta)=0,~\xi^T \nabla^2 U(\theta)\xi=0 ~\Longrightarrow ~\xi=0,\quad \forall \theta \in \operatorname{int}(\M).
\end{align}
\end{lemma} 
\begin{proof} 
For any $\theta \in \operatorname{int}(\M)$, and $\xi \in \R^n$, based on the properties of $K$, we have 
\[
\xi^T \nabla^2 h(\theta) \xi =K''(U(\theta))|\nabla U(\theta) \xi|^2 + K'(U(\theta)) \xi^T \nabla^2 U(\theta) \xi^T \geq 0.  
\]
Equality holds if and only if 
\[
\xi\cdot \nabla U(\theta)=0 \quad \text{and}\;\;  
\xi^T \nabla^2 U(\theta)\xi=0.  
\]
Hence, \eqref{con} implies the strict convexity of $h$. 
\end{proof}

\begin{remark}
When $K(U)$ does not satisfy (\ref{con}), adding a correction term  is sufficient to ensure the strict convexity of  $h$ in $\M$. For instance, for the domain
\begin{equation}\label{dm}
\M=\{\theta\in\R^n\;|\;U_i(\theta)\geq 0, i\in[p]\}
\end{equation}
with each $U_i$ being concave, a valid choice is 
\[
h(\theta)=\sum_{i=1}^p K(U_i(\theta))+\tilde h(\theta),\quad \theta \in \M,   
\]
where $\tilde h$ is added to ensure the strict convexity of $h$ in $\M$. If $\tilde h$ is strictly convex, then 
\begin{equation*}
\begin{split}
    \xi^T \nabla^2 h(\theta) \xi =&\sum_{i=1}^p \left(K''(U_i(\theta))|\nabla U_i(\theta) \xi|^2 +  K'(U_i(\theta)) \xi^T \nabla^2 U_i(\theta) \xi^T \right) + \xi^T \nabla^2 \tilde{h}(\theta) \xi\\
    \geq& \xi^T \nabla^2 \tilde{h}(\theta) \xi>0,\quad \forall \xi\neq 0.
\end{split}
\end{equation*}
Non-strictly convex alternatives  for $\tilde h$ are  viable, provided they result in a strongly convex function $h$ when combined with $(U_i)_{i=1}^p$.
\begin{example}
We illustrate several specific examples below.
\begin{enumerate}
    \item When $U$ is strictly concave for the domain $\M=\{\theta\in\R^n\;|\; U(\theta) \geq 0\}$,  with $U=1-\theta^T S \theta$, where $S$ symmetric and positive definite, a suitable choice for $h$ is given by   
\begin{equation}\label{h-cir}
h(\theta)=K(1-\theta^T S \theta). 
\end{equation}
\item In case where $\nabla U$ spans the entire space $\mathbb{R}^n$ for 
the domain 
$\M=\{\theta\in\R^n\;|\; a_i\leq \theta_i, i\in[n]\},$
a suitable option is  
\begin{equation}\label{h-b}
h(\theta)=\sum_{i=1}^n K(\theta_i-a_i).     
\end{equation}
\item For the domain $\M = \{\theta\in \mathbb{R}^n~|~a_i \leq \theta_i, i\in [p]\}$, where $1\leq p<n$, a valid choice for $h$ is gen by  
\begin{equation}\label{th}
h(\theta) =\sum_{i=1}^p K(\theta_i-a_i)+\frac{1}{2}\sum_{i=p+1}^n\theta_i^2 ,
\end{equation}
is a valid choice. 
\end{enumerate}
\end{example}
\end{remark}

\section{On the projected PL condition}\label{pplfunc}
For the constrained minimization problem:
\begin{align*}
\min\quad  & L(\theta)=\frac{1}{2}\left(\beta \theta_1^2+\alpha \theta_2^2 \right),  \\
\text{subject to}\quad  & q(\theta)=a\theta_1+b\theta_2 -1=0, 
\end{align*}
where $ab\not=0, \alpha \geq \beta >0$. The minimum point of this problem is given by 
$$
\theta^* = \bigg(\frac{a\alpha}{a^2\alpha+b^2\beta}, \frac{b\beta}{a^2\alpha+b^2\beta}\bigg).
$$
Using $\nabla q(\theta)=(a,b)$ and Equation~\eqref{AGPk+},  the projection matrix is 
$$
P = \frac{1}{a^2+b^2}
\begin{pmatrix}
b^2 & -ab\\
-ab & a^2
\end{pmatrix},
$$
and the projected gradient is 
$$
P^\top \nabla L(\theta) = \frac{b\beta\theta_1-a\alpha\theta_2}{a^2+b^2}(b, -a)^\top.
$$
To verify the projected PL condition, it is observed that
\begin{align*}
\|P^\top \nabla L(\theta)\|^2 
&= \frac{(b\beta\theta_1-a\alpha\theta_2)^2}{a^2+b^2}\\
&= \frac{1}{a^2+b^2}\bigg((b\beta+\frac{a^2\alpha}{b})\theta_1-\frac{a\alpha}{b}\bigg)^2\\
& = \frac{1}{(a^2+b^2)b^2}\bigg((b^2\beta+a^2\alpha)\theta_1-a\alpha\bigg)^2\\
&=\frac{(b^2\beta+a^2\alpha)^2}{(a^2+b^2)b^2}(\theta_1-\theta^*_1)^2,    
\end{align*}
where $\theta_2=(1-a\theta_1)/b$ was used in the second equality. Also, 
\begin{align*}
 L(\theta)-L(\theta^*) 
 & = \frac{1}{2}(\beta\theta^2_1+\alpha\theta^{2}_2) -\frac{1}{2}(\beta\theta^{*2}_1+\alpha\theta^{*2}_2)\\
 & = \frac{1}{2}\beta(\theta^2_1-\theta^{*2}_1) + \frac{1}{2}\alpha(\theta^2_2-\theta^{*2}_2) \\
 & = \frac{1}{2}\beta(\theta^2_1-\theta^{*2}_1) + \frac{1}{2}\alpha \bigg((\frac{1-a\theta_1}{b})^2-(\frac{1-a\theta^*_1}{b})^2\bigg)\\
 & =  \bigg(\frac{1}{2}\beta+\frac{1}{2}\alpha \frac{a^2}{b^2}\bigg)(\theta^2_1 - \theta^{*2}_1) - \frac{a\alpha}{b^2}(\theta_1 - \theta^{*}_1)\\
 & = \frac{a^2\alpha+b^2\beta}{2b^2}(\theta_1 - \theta^{*}_1)\bigg((\theta_1 + \theta^{*}_1)-\frac{2a\alpha}{a^2\alpha+b^2\beta}\bigg)\\
 & = \frac{a^2\alpha+b^2\beta}{2b^2}(\theta_1 - \theta^{*}_1)^2.
\end{align*}
Hence the projected PL condition holds for $\mu>0$ if 
$$
\mu \leq \frac{a^2\alpha+b^2\beta}{a^2+b^2}.
$$

\medskip

\bibliographystyle{amsplain}
\bibliography{ref}

\end{document}